\documentclass[a4paper, 12pt]{article}

\usepackage[latin1]{inputenc}
\usepackage[T1]{fontenc}

\usepackage[dvips]{graphicx}
\usepackage{color}
\usepackage{amsthm}  
\usepackage{amsmath} 
\usepackage{amssymb} 
\usepackage{latexsym} 

\theoremstyle{plain}

\newtheorem{proposition}{Proposition}

\newtheorem{theorem}{Theorem}
\newtheorem{corollary}{Corollary}

\newtheorem{lemma}{Lemma}

\newtheorem{conjecture}{Conjecture}

\newcommand{\mc}[1]{\mathcal{ #1 }}
\newcommand{\tn}[1]{\textnormal{#1}}

\usepackage[numbers]{natbib}
\usepackage{subfigure}

\title{Double-critical graphs and complete minors}

\author{ \small Ken-ichi Kawarabayashi (\texttt{k\_keniti@nii.ac.jp}) \\
   \small The National Institute of Informatics \\
   \small 2-1-2 Hitotsubashi, Chiyoda-ku \\
   \small Tokyo 101-8430, Japan \\
   \vspace{5mm} \and
   \small Anders Sune Pedersen (\texttt{asp@imada.sdu.dk}) \\
   \small Bjarne Toft (\texttt{btoft@imada.sdu.dk}) \\
   \small Dept. of Mathematics \& Computer Science \\
   \small University of Southern Denmark \\
  \small Campusvej 55, 5230 Odense M, Denmark \vspace{5mm} \\
  \small MR Subject Classification: 05C15, 05C69}

\begin{document}

\maketitle
\begin{abstract}
  A connected $k$-chromatic graph $G$ is double-critical if for all
  edges $uv$ of $G$ the graph $G - u - v$ is $(k-2)$-colourable. The
  only known double-critical $k$-chromatic graph is the complete
  $k$-graph $K_k$. The conjecture that there are no other
  double-critical graphs is a special case of a conjecture from 1966,
  due to Erd\H{o}s and Lov\'asz.  The conjecture has been verified for
  $k \leq 5$.  We prove for $k=6$ and $k=7$ that any non-complete
  double-critical $k$-chromatic graph is $6$-connected and has $K_k$
  as a minor.
\end{abstract}

\section{Introduction}
A long-standing conjecture, due to Erd\H{o}s and
Lov\'asz~\cite{TihanyProblem2}, states that the complete graphs are
the only double-critical graphs. We refer to this conjecture as the
\emph{Double-Critical Graph Conjecture}. A more elaborate statement of
the conjecture is given in Section~\ref{sec:notation}, where several
other fundamental concepts used in the present paper are defined. The
Double-Critical Graph Conjecture is easily seen to be true for
double-critical $k$-chromatic graphs with $k \leq 4$.
Mozhan~\cite{0688.05026} and Stiebitz~\cite{MR1221590, MR882614}
independently proved the conjecture for $k = 5$, while it remains open
for $k \geq 6$. The Double-Critical Graph Conjecture is a special case
of a more general conjecture, the so-called Erd\H{o}s-Lov\'asz Tihany
Conjecture~\cite{TihanyProblem2}, which states that for any graph $G$
with $\chi (G) > \omega (G)$ and any two integers $a, b \geq 2$ with
$a+b = \chi (G) + 1$, there is a partition $(A,B)$ of the vertex set
$V(G)$ such that $\chi( G[A] ) \geq a$ and $\chi( G[B] ) \geq b$. The
Erd\H{o}s-Lov\'asz Tihany Conjecture is settled in the affirmative for
every pair $(a,b) \in \{ (2,2), (2,3), (2,4), (3,3), (3,4), (3,5) \}$
(see \cite{MR0242718, 0688.05026, MR1221590, MR882614}). Kostochka and
Stiebitz~\cite{KostochkaStiebitz2008} proved it to be true for line
graphs of multigraphs, while Balogh et al.~\cite{Balogh08} proved it
to be true for quasi-line graphs and for graphs with independence
number $2$.

In addition, Stiebitz (private communication) has proved a weakening
of the Erd\H{o}s-Lov\'asz Tihany conjecture, namely that for any graph
$G$ with $\chi (G) > \omega (G)$ and any two integers $a,b \geq 2$
with $a + b = \chi (G) + 1$, there are two disjoint subsets $A$ and
$B$ of the vertex set $V(G)$ such that $\delta ( G[A] ) \geq a - 1$
and $\delta ( G[B] ) \geq b-1$. (Note that for this conclusion to hold
it is not enough to assume that $G \nsubseteq K_{a+b-1}$ and $\delta
(G) \geq a+b-2$, that is, the Erd\H{o}s-Lov\'asz Tihany conjecture
does not hold in general for the so-called colouring number. The
$6$-cycle with all shortest diagonals added is a counterexample with
$a=2$ and $b=4$.) For $a=2$, the truth of this weaker version of the
Erd\H{o}s-Lov\'asz Tihany conjecture follows easily from
Theorem~\ref{th:noAdjacentLowVertices} of the present paper.

Given the difficulty in settling the Double-Critical Graph Conjecture
we pose the following weaker conjecture:
\begin{conjecture}
  Every double-critical $k$-chromatic graph is contractible
  to the complete $k$-graph.
\label{conj:relaxedDC}
\end{conjecture}
Conjecture~\ref{conj:relaxedDC} is a weaker version of Hadwiger's
Conjecture~\cite{MR0012237}, which states that every $k$-chromatic
graph is contractible to the complete $k$-graph.  Hadwiger's
Conjecture is one of the most fundamental conjectures of Graph Theory,
much effort has gone into settling it, but it remains open for $k \geq
7$. For more information on Hadwiger's Conjecture and related problems
we refer the reader to~\cite{JensenToft95, MR1411244}.

In this paper we mainly devote attention to the
\emph{double-critical} $7$-chromatic graphs. It seems that relatively
little is known about $7$-chromatic graphs.  Jakobsen~\cite{MR0340108}
proved that every $7$-chromatic graph has a $K_7$ with two edges
missing as a minor. It is apparently not known whether every
$7$-chromatic graph is contractible to $K_7$ with one edge
missing. Kawarabayashi and Toft~\cite{MR2141662} proved that every
$7$-chromatic graph is contractible to $K_7$ or $K_{4,4}$.

The main result of this paper is that any double-critical $6$- or
$7$-chromatic graph is contractible to the complete graph on six or
seven vertices, respectively. These results are proved in
Sections~\ref{sec:dc6chrom}~and~\ref{sec:dc7chrom} using results of
Gy{\H{o}}ri~\cite{MR652887} and Mader~\cite{MR0229550}, but not the
Four Colour Theorem. Krusenstjerna-Hafstr{\o}m and
Toft~\cite{MR634546} proved that any double-critical $k$-chromatic ($k
\geq 5$) non-complete graph is $5$-connected and
$(k+1)$-edge-connected. In Section~\ref{sec:connectivity} we extend
that result by proving that any double-critical $k$-chromatic ($k \geq
6$) non-complete graph is $6$-connected. In Section~\ref{sec:basic} we
exhibit a number of basic properties of double-critical non-complete
graphs. In particular, we observe that the minimum degree of any
double-critical non-complete $k$-chromatic graph $G$ is at least $k+1$
and that no two vertices of degree $k+1$ are adjacent in $G$, cf.
Proposition~\ref{prop:minimumDegree} and
Theorem~\ref{th:noAdjacentLowVertices}. Gallai~\cite{MR0188100} also
used the concept of decomposable graphs in the study of critical
graphs. In Section~\ref{sec:decomposable} we use double-critical
decomposable graphs to study the maximum ratio between the number of
double-critical edges in a non-complete critical graph and the size of
the graph, in particular, we prove that, for every non-complete
$4$-critical graph $G$, this ratio is at most $1/2$ and the maximum is
attained if and only if $G$ is a wheel. Finally, in
Section~\ref{sec:mixed}, we study two variations of the concept of
double-criticalness, which we have termed double-edge-criticalness and
mixed-double-criticalness. It turns out to be straightforward to show
that the only double-edge-critical graphs and mixed-double-critical
graphs are the complete graphs.
\section{Notation}
\label{sec:notation}
All graphs considered in this paper are simple and finite. We let
$n(G)$ and $m(G)$ denote the order and size of a graph $G$,
respectively. The path, the cycle and the complete graph on $n$
vertices is denoted $P_n$, $C_n$ and $K_n$, respectively. The
\emph{length} of a path or a cycle is its number of edges. The set of
integers $\{1,2, \ldots, k \}$ will be denoted $[k]$. A
\emph{$k$-colouring} of a graph $G$ is a function $\varphi$ from the
vertex set $V(G)$ of $G$ into a set $\mc{C}$ of cardinality $k$ so
that $\varphi(u) \neq \varphi (v)$ for every edge $uv \in E(G)$, and a
graph is \emph{$k$-colourable} if it has a $k$-colouring.  The
elements of the set $\mc{C}$ is referred to as colours, and a vertex
$v \in V(G)$ is said to be assigned the colour $\varphi (v)$ by $\phi$.  The set
of vertices $S$ assigned the same colour $c \in \mc{C}$ is said to
constitute the colour class $c$. The minimum integer $k$ for which a
graph $G$ is $k$-colourable is called its \emph{chromatic number} of
$G$ and it is denoted $\chi (G)$. An \emph{independent set} $S$ of
$V(G)$ is a set such that the induced graph $G[S]$ is edge-empty. The
maximum integer $k$ for which there exists an independent set $S$ of
$G$ of cardinality $k$ is the \emph{independence number} of $G$ and is
denoted $\alpha (G)$. A graph $H$ is a \emph{minor} of a graph $G$ if
$H$ can be obtained from $G$ by deleting edges and/or vertices and
contracting edges. An \emph{$H$-minor} of $G$ is a minor of $G$
isomorphic to $H$. Given a graph $G$ and a subset $U$ of $V(G)$ such
that the induced graph $G[U]$ is connected, the graph obtained
from $G$ by contracting $U$ into one vertex is denoted $G/U$, and the
vertex of $G/U$ corresponding to the set $U$ of $G$ is denoted $v_U$.
Let $\delta (G)$ denote the minimum degree of $G$. For a vertex $v$ of
a graph $G$, the (open) neighbourhood of $v$ in $G$ is denoted
$N_G(v)$ while $N_G[v]$ denotes the closed neighbourhood $N_G(v) \cup
\{ v \}$.  Given two subsets $X$ and $Y$ of $V(G)$, we denote by
$E[X,Y]$ the set of edges of $G$ with one end-vertex in $X$ and the
other end-vertex in $Y$, and by $e(X,Y)$ their number. If $X = Y$,
then we simply write $E(X)$ and $e(X)$ for $E[X,X]$ and $e(X,X)$,
respectively. The induced graph $G[N(v)]$ is refered to as the
neighbourhood graph of $v$ (w.r.t. $G$) and it is denoted $G_v$.  The
independence number $\alpha (G_v)$ is denoted $\alpha_v$. The degree
of a vertex $v$ in $G$ is denoted $\deg_G (v)$ or $\deg (v)$. A graph
$G$ is called \emph{vertex-critical} or, simply, \emph{critical} if
$\chi (G - v) < \chi(G)$ for every vertex $v \in V(G)$. Moreover, a
critical graph $G$ is called \emph{double-critical} if
\begin{equation}
\chi (G - x - y) \leq \chi(G) - 2 \textrm{ for all edges } xy \in E(G)
\label{eq:928568345}
\end{equation}
It is clear that $\chi (G - x - y)$ can never be strictly less than
$\chi (G) - 2$ and so we could require $\chi (G - x - y ) = \chi(G) -
2$ in~\eqref{eq:928568345}. The fact that any double-critical graph is
vertex-critical implies that any double-critical graph is connected.
The concept of vertex-critical graphs was first introduced by Dirac
\cite{MR0045371} and have since been studied extensively, see, for
instance, \cite{JensenToft95}. As noted by Dirac~\cite{MR0045371},
every critical $k$-chromatic graph $G$ has minimum degree $\delta (G)
\geq k-1$. An edge $xy \in E(G)$ such that $\chi (G - x - y) = \chi
(G) - 2$ is referred to as a \emph{double-critical edge}. For
graph-theoretic terminology not explained in this paper, we refer the
reader to~\cite{MR2368647}.
\section{Basic properties of non-complete double-critical graphs}
\label{sec:basic}
In this section we let $G$ denote a non-complete double-critical $k$-chromatic
graph. Thus, by the aforementioned results, $k \geq 6$.
\begin{proposition}
The graph $G$ does not contain a complete $(k-1)$-graph as a subgraph.
\label{prop:forbiddenCompleteKminusOne}
\end{proposition}
\begin{proof}
  Suppose $G$ contains $K_{k-1}$ as a subgraph. Since $G$ is
  double-critical and $k$-chromatic, it follows that $G - K_{k-1}$ is
  edge-empty, but not vertex-empty. Since $G$ is also vertex-critical,
  $\delta(G) \geq k-1$, and therefore every $v \in V(G-K_{k-1})$ is
  adjacent to every vertex of $V(K_{k-1})$ in $G$, in particular, $G$
  contains $K_k$ as a subgraph. Since $G$ is vertex-critical, $G =
  K_k$, a contradiction.
\end{proof}
\begin{proposition}
  If $H$ is a connected subgraph of $G$ with $n(H) \geq 2$, then the
  graph $G/V(H)$ obtained from $G$ by contracting $H$ is
  $(k-1)$-colourable.
\end{proposition}
\begin{proof}
  The graph $H$ contains at least one edge $uv$, and the graph $G - u
  - v$ is $(k-2)$-colourable, which, in particular, implies that the
  graph $G-H$ is $(k-2)$-colourable. Now, any $(k-2)$-colouring of
  $G-H$ may be extended to a $(k-1)$-colouring of $G/V(H)$ by
  assigning a new colour to the vertex $v_{V(H)}$.
\end{proof}
Given any edge $xy \in E(G)$, define 
\begin{eqnarray*}
  A(xy) & := & N(x) \backslash N[y] \\
  B(xy) & := & N(x) \cap N(y) \\
  C(xy) & := & N(y) \backslash N[x] \\
  D(xy) & := & V(G) \backslash ( N(x) \cup N(y) ) \\
  &  = & V(G) \backslash \left( A(xy) \cup B(xy) \cup C(xy) \cup \{ x, y \} \right)
\end{eqnarray*}
We refer to $B(x,y)$ as the \emph{common neighbourhood} of $x$ and $y$
(in $G$).

In the proof of Proposition~\ref{prop:CyclesFromGenKempeChains} we use
what has become known as generalized Kempe chains, cf.~\cite{MR693366,
  MR1373659}. Given a $k$-colouring $\varphi$ of a graph $H$, a
vertex $x \in H$ and a permutation $\pi$ of the colours $1, 2, \ldots,
k$. Let $N_1$ denote the set of neighbours of $x$ of colour $\pi
(\varphi(x))$, let $N_2$ denote the set of neighbours of $N_1$ of
colour $\pi (\pi (\varphi(x)))$, let $N_3$ denote the set of
neighbours of $N_2$ of colour $\pi^3 (\varphi(x))$, etc. We call $N(x,
\varphi, \pi) = \{x \} \cup N_1 \cup N_2 \cup \cdots$ \emph{a
  generalized Kempe chain from $x$ w.r.t.\ $\varphi$ and $\pi$}.
Changing the colour $\varphi(y)$ for all vertices $y \in N(x, \varphi,
\pi)$ from $\varphi(y)$ to $\pi (\varphi(y))$ gives a new
$k$-colouring of $H$.
\begin{proposition}
  For all edges $xy \in E(G)$, $(k-2)$-colourings of $G - x - y$ and
  any non-empty sequence $j_1, j_2, \ldots, j_i$ of $i$ different
  colours from $[k-2]$, there is a path of order $i+2$ starting at
  $x$, ending at $y$ and with the $t$'th vertex after $x$ having
  colour $j_t$ for all $t \in [i]$. In particular, $xy$ is contained
  in at least $(k-2)! / (k-2-i)!$ cycles of length $i+2$.
\label{prop:CyclesFromGenKempeChains}
\end{proposition}
\begin{proof}
  Let $xy$ denote an arbitrary edge of $G$ and let $\varphi$ denote a
  $(k-2)$-colouring of $G - x - y$ which uses the colours of $[k-2]$.
  The function $\varphi$ is extended to a proper $(k-1)$-colouring of
  $G - xy$ by defining $\varphi(x) = \varphi(y) = k-1$. Let $\pi$
  denote the cyclic permutation $(k-1, j_1, j_2, \ldots, j_i)$. If the
  generalized Kempe chain $N(x, \varphi, \pi)$ does not contain the
  vertex $y$, then by reassigning colours on the vertices of $N(x,
  \varphi, \pi)$ as described above, a $(k-1)$-colouring $\psi$ of $G
  - xy$ with $\psi(x) \neq k-1 = \psi (y)$ is obtained, contradicting
  the fact that $G$ is $k$-chromatic. Thus, the generalized Kempe
  chain $N(x, \varphi, \pi)$ must contain the vertex $y$. Since $x$
  and $y$ are the only vertices which are assigned the colour $k-1$ by
  $\varphi$, it follows that the induced graph $G[N(x, \varphi, \pi)]$
  contains an $(x,y)$-path of order $i+2$ with vertices coloured
  consecutively $k-1, j_1, j_2, \ldots, j_i, k-1$. The last claim of
  the proposition follows from the fact there are
  $(k-2)! / (k-2-i)!$ ways of selecting and ordering $i$ elements from the
  set $[k-2]$.
\end{proof}
Note that the number of cycles of a given length obtained in
Proposition~\ref{prop:CyclesFromGenKempeChains} is exactly the number
of such cycles in the complete $k$-graph. Moreover,
Proposition~\ref{prop:CyclesFromGenKempeChains} immediately implies
the following result.
\begin{corollary}
  For all edges $xy \in E(G)$ and $(k-2)$-colourings of $G - x - y$,
  the set $B(xy)$ of common neighbours of $x$ and $y$ in $G$ contains
  vertices from every colour class $i \in [k-2]$, in particular,
  $|B(xy)| \geq k-2$, and $xy$ is contained in at least $k-2$
  triangles.
\label{cor:atLeastKminusTwoCommonNeighbours}
\end{corollary}

\begin{proposition}
  For all vertices $x \in V(G)$, the minimum degree in the induced graph
  of the neighbourhood of $x$ in $G$ is at least $k-2$, that is, $\delta
  (G_x) \geq k-2$.
\label{prop:minimumDegreeInGx}
\end{proposition}
\begin{proof}
  According to
  Corollary~\ref{cor:atLeastKminusTwoCommonNeighbours}, $|B(xy)|
  \geq k-2$ for any vertex $y \in N(x)$, which implies that $y$ has at
  least $k-2$ neighbours in $G_x$.
\end{proof}
\begin{proposition}
  For any vertex $x \in V(G)$, there exists a vertex $y \in N(x)$ such
  that the set $A(xy)$ is not empty.
\label{prop:AxyNonEmpty}
\end{proposition}
\begin{proof}
  Let $x$ denote any vertex of $G$, and let $z$ in $N(x)$. The common
  neighbourhood $B(xz)$ contains at least $k-2$ vertices, and so,
  since $K_{k-1}$ is not a subgraph of $G$, not every pair of vertices
  of $B(xy)$ are adjacent, say $y, y' \in B(xz)$ are non-adjacent. Now
  $y' \in A(xy)$, in particular, $A(xy)$ is not empty.
\end{proof}
\begin{proposition}
  There exists at least one edge $xy \in E(G)$ such that the set
  $D(xy)$ is not empty.
\label{prop:DnotEmpty}
\end{proposition}
\begin{proof}
  According to Proposition~\ref{prop:AxyNonEmpty}, there exists at
  least one edge $uv \in E(G)$ such that $A(uv)$ is non-empty. Fix a
  vertex $a \in A(uv)$. This vertex $a$ cannot be adjacent to every
  vertex of $B(uv)$, since that, according to
  Corollary~\ref{cor:atLeastKminusTwoCommonNeighbours}, would leave
  no colour available for $a$ in a $(k-2)$-colouring of $G -
  u - v$. Suppose $a$ is not adjacent to $z \in B(uv)$. Now $a \in
  D(vz)$, in particular, $D(vz)$ is not empty.
\end{proof}
\begin{proposition}
  If $A(xy)$ is non-empty for some edge $xy \in E(G)$, then $\delta
  (G[A(xy)]) \geq 1$, that is, the induced subgraph $G[A(xy)]$ contains
  no isolated vertices. By symmetry, $\delta (G[C(xy)]) \geq 1$, if
  $C(xy)$ is non-empty.
\label{prop:noIsolatedInAxy}
\end{proposition}
\begin{proof}
  Suppose $G[A(xy)]$ contains some isolated vertex, say $a$. Now,
  since $G$ is double-critical, $|B(xa)| \geq k-2$, and, since $a$ is
  isolated in $A(xy)$, the common neighbours of $x$ and $a$ must lie
  in $B(xy)$, in particular, any $(k-2)$-colouring of $G - a - x$ must
  assign all colours of the set $[k-2]$ to common neighbours of $a$
  and $x$ in $B(xy)$. But this leaves no colour in the set $[k-2]$
  available for $y$, which contradicts the fact that $G - a - x$ is
  $(k-2)$-colourable. This contradiction implies that $G[A(xy)]$
  contains no isolated vertices.
\end{proof}
\begin{proposition}
  If some vertex $y \in N(x)$ is not adjacent to some vertex $z \in
  N(x) \backslash \{y\}$, then there exists another vertex $w \in N(x)
  \backslash \{y,z \}$, which is also not adjacent to
  $y$. Equivalently, no vertex of the complement $\overline{G_x}$ has
  degree $1$ in $\overline{G_x}$.
\label{prop:missingNeighboursInGx}
\end{proposition}
\begin{proof}
  This follows directly from Proposition~\ref{prop:noIsolatedInAxy}.
  If $y \in N(x)$ is not adjacent to $z \in N(x) \backslash \{y\}$,
  then $z \in A(xy)$ and, since $G[A(xy)]$ contains no isolated
  vertices, the set $A(xy) \backslash \{ z \}$ cannot be empty. 
\end{proof}
\begin{proposition}
  Every vertex of $G$ has at least $k+1$ neighbours.
\label{prop:minimumDegree}
\end{proposition}
\begin{proof}
  According to Proposition~\ref{prop:AxyNonEmpty}, for any vertex $x
  \in V(G)$, there exists a vertex $y \in N(x)$ such that $A(xy) \neq
  \emptyset$, and, according to
  Proposition~\ref{prop:noIsolatedInAxy}, $\delta(G[A(xy)]) \geq 1$,
  in particular, $|A(xy)| \geq 2$. Since $N(x)$ is the union of the
  disjoint sets $A(xy)$, $B(xy)$ and $\{ y \}$, we obtain
$$\deg_G (x) = |N(x)| \geq |A(xy)| + |B(xy)| + 1 \geq 2 + (k-2) + 1 = k+1$$
where we used the fact that $|B(xy)| \geq k-2$, according to
Corollary~\ref{cor:atLeastKminusTwoCommonNeighbours}.
\end{proof}
\begin{proposition}
  For any vertex $x \in V(G)$,
\begin{equation}
  \deg_G (x) - \alpha_x \geq |B(xy)| + 1 \geq k - 1
\label{eq:43658721084}
\end{equation}
where $y \in N(x)$ is any vertex contained in an independent set in
$N[x]$ of size $\alpha_x$. Moreover, $\alpha_x \geq 2$.
\label{prop:upperBoundOnAlphaGx}
\end{proposition}
\begin{proof}
  Let $S$ denote an independent set in $N(x)$ of size $\alpha_x$.
  Obviously, $\alpha_x \geq 2$, otherwise $G$ would contain a $K_k$.
  Choose some vertex $y \in S$. Now the non-empty set $S \backslash \{
  y \}$ is a subset of $A(xy)$, and, according to
  Proposition~\ref{prop:noIsolatedInAxy}, $\delta (G[A(xy)]) \geq 1$.
  Let $a_1$ and $a_2$ denote two neighbouring vertices of $A(xy)$. The
  independet set $S$ of $G_x$ contains at most one of the vertices
  $a_1$ and $a_2$, say $a_1 \notin S$. Therefore $S$ is a subset of
  $\{ y \} \cup A(xy) \backslash \{ a_1 \}$, and so we obtain
$$\alpha_x \leq |A(xy)| = |N(x)| - |B(xy)| - 1 \leq \deg_G (x) - (k-2) - 1$$
from which~\eqref{eq:43658721084} follows.
\end{proof}
\begin{proposition}
  For any vertex $x$ not joined to all the other vertices of $G$,
  $\chi (G_x) \leq k-3$.
\label{prop:kMinusTwoColourableNeighbourGraphGx}
\end{proposition}
\begin{proof} Since $G$ is connected there must be some vertex, say
  $z$, in $V(G) \backslash N[x]$, which is adjacent to some vertex,
  say $y$, in $N(x)$. Now, clearly, $z$ is a vertex of $C(xy)$, in
  particular, $C(xy)$ is not empty, which, according to
  Proposition~\ref{prop:noIsolatedInAxy}, implies that $C(xy)$
  contains at least one edge, say $e = zv$. Since $G$ is
  double-critical, it follows that $\chi (G - z - v ) \leq k-2$, in
  particular, the subgraph $G[N[x]]$ of $G - z - v$ is
  $(k-2)$-colourable, and so $G_x$ is $(k-3)$-colourable.
\end{proof}

\begin{proposition}
  If $\deg_G (x) = k+1$, then the complement $\overline{G_x}$ consists
  of isolated vertices (possibly none) and cycles (at least one),
  where the length of the cycles are at least five.
\label{prop:structureOfGx}
\end{proposition}
\begin{proof}
  Given $\deg_G (x) = k + 1$, suppose that some vertex $y \in G_x$ has
  three edges missing in $G_x$, say $yz_1, yz_2, yz_3$. Now $B(xy)$ is
  a subset of $N(x)\backslash \{y, z_1, z_2, z_3 \}$.  However,
  $|N(x)\backslash \{y, z_1, z_2, z_3 \}| = (k+1) - 4$, which implies
  $|B(xy)| \leq k-3$, contrary to
  Corollary~\ref{cor:atLeastKminusTwoCommonNeighbours}. Thus no vertex
  of $G_x$ is missing more than two edges. According to
  Proposition~\ref{prop:noIsolatedInAxy}, if a vertex of $G_x$ is
  missing one edge, then it is missing at least two edges. Thus, it
  follows that $\overline{G_x}$ consists of isolated vertices and
  cycles. If $\overline{G_x}$ consists of only isolated vertices, then
  $G_x$ would be a complete graph, and $G$ would contain a complete
  $(k+1)$-graph, contrary to our assumptions.  Thus, $\overline{G_x}$
  contains at least one cycle $C$. Let $s$ denote a vertex of $C$, and
  let $r$ and $t$ denote the two distinct vertices of $A(xs)$. Now $G
  - x - s$ is $(k-2)$-colourable and, according to
  Corollary~\ref{cor:atLeastKminusTwoCommonNeighbours}, each of the
  $k-2$ colours is assigned to at least one vertex of the common
  neighbourhood $B(xs)$. Thus, both $r$ and $t$ must have at least one
  non-neighbour in $B(xs)$, and, since $r$ and $t$ are adjacent, it
  follows that $r$ and $t$ must have distinct non-neighbours, say $q$
  and $u$, in $B(xs)$. Now, $q,r,s,t$ and $u$ induce a path of length
  four in $\overline{G_x}$ and so the cycle $C$ containing $P$ has
  length at least five.
\end{proof}

\begin{theorem}
  No two vertices of degree $k+1$ are adjacent in $G$.
\label{th:noAdjacentLowVertices}
\end{theorem}
\begin{proof}
  Firstly, suppose $x$ and $y$ are two adjacent vertices of degree
  $k+1$ in $G$. Suppose that the one of the sets $A(xy)$ and $C(xy)$
  is empty, say $A(xy) = \emptyset$. Then $|B(xy)| = k$ and $C(xy) =
  \emptyset$. Obviously, $\alpha_x \geq 2$, and it follows from
  Proposition~\ref{prop:upperBoundOnAlphaGx} that $\alpha_x$ is equal
  to two. Let $\varphi$ denote a $(k-2)$-colouring of $G - x - y$. Now
  $|B(xy)| = k$, $\alpha_x = 2$ and the fact that $\varphi$ applies
  each colour $c \in [k-2]$ to at least one vertex of $B(xy)$ implies
  that exactly two colours $i, j \in [k-2]$ are applied twice among the
  vertices of $B(xy)$, say $\varphi( u_1) = \varphi (u_2) = k-3$ and
  $\varphi (v_1) = \varphi (v_2) = k-2$, where $u_1, u_2, v_1$ and
  $v_2$ denotes four distinct vertices of $B(xy)$. Now each of the
  colours $1, \ldots, k-4$ appears exactly once in the colouring of
  the vertices of $W := B(xy) \backslash \{ u_1, u_2, v_1, v_2 \}$,
  say $W = \{ w_1, \ldots, w_{k-4} \}$ and $\varphi (w_i) = i$ for
  each $i \in [k-4]$. Now it follows from
  Proposition~\ref{prop:CyclesFromGenKempeChains} that there exists a
  path $xw_i w_jy$ for each pair of distinct colours $i, j \in [k-4]$.
  Therefore $G[W] = K_{k-4}$. If one of the vertices $u_1, u_2, v_1$
  or $v_2$, say $u_1$, is adjacent to every vertex of $W$, then $G[W
  \cup \{u_1, x, y \}] = K_{k-1}$, which contradicts
  Proposition~\ref{prop:forbiddenCompleteKminusOne}. Hence each of the
  vertices $u_1, u_2, v_1$ and $v_2$ is missing at least one neighbour
  in $W$. It follows from Proposition~\ref{prop:structureOfGx}, that
  the complement $\overline{G[B(xy)]}$ consists of isolated vertices
  and cycles of length at least five. Now it is easy to see that
  $\overline{G[B(xy)]}$ contains exactly one cycle, and we may
  w.l.o.g.\ assume that $u_1w_1v_1v_2w_2u_2$ are the vertices of that
  cycle. Now $G[\{ u_1, v_1 \} \cup W \backslash \{ w_1 \}] =
  K_{k-1}$, and we have again obtained a contradiction.

  Secondly, suppose that one of the sets $A(xy)$ and $C(xy)$ is non-empty,
  say $A(xy) \neq \emptyset$.  Since, according to
  Corollary~\ref{cor:atLeastKminusTwoCommonNeighbours}, the common
  neighbourhood $B(xy)$ contains at least $k-2$ vertices, it follows
  from Proposition~\ref{prop:noIsolatedInAxy} that $|A(xy)|=2$ and so
  $|B(xy)|=k-2$, which implies $|C(xy)| = 2$. Suppose $A(xy) = \{ a_1,
  a_2 \}$, $C(xy) = \{ c_1, c_2 \}$, and let $C_A$ denote the cycle of
  the complement $\overline{G_x}$ which contains the vertices $a_1$,
  $y$ and $a_2$, say $C_A = a_1 y a_2 u_1 \ldots u_i$, where $u_1,
  \ldots, u_i \in B(xy)$ and $i \geq 2$.  Similarly, let $C_C$ denote
  the cycle of the complement $\overline{G_y}$ which contains the
  vertices $c_1$, $x$ and $c_2$, say $C_A = c_1 x c_2 v_1 \ldots v_j$,
  where $v_1, \ldots, v_j \in B(xy)$ and $j \geq 2$. Since both
  $\overline{G_x}$ and $\overline{G_y}$ consists of only isolated
  vertices (possibly none) and cycles, it follows that we must have
  $(u_1, \ldots, u_i) = (v_1, \ldots, v_j)$ or $(u_1, \ldots, u_i) =
  (v_j, \ldots, v_j)$. We assume w.l.o.g.\ that the former holds.

  Let $\varphi$ denote some $(k-2)$-colouring of $G - x - y$ using the
  colours of $[k-2]$, and suppose w.l.o.g.\ $\phi (a_1) = k-2$ and
  $\varphi (a_2) = k-3$. Again, the structure of $\overline{G_x}$ and
  $\overline{G_y}$ implies $\varphi (u_1) = k-3$ and $\varphi (u_i) =
  k-2$, which also implies $\varphi (c_1) = k-2$ and $\varphi (c_2)= k-3$.

  Let $U = B(xy) \backslash \{ u_1, u_i \}$. Now $U$ has size $k-4$
  and precisely one vertex of $U$ is assigned the colour $i$ for each
  $i \in [k-4]$. Since no other vertices of $( N(x) \cup N(y) )
  \backslash U$ is assigned a colour from the set $[k-4]$, it follows
  from Proposition~\ref{prop:CyclesFromGenKempeChains} that for each
  pair of distinct colours $s, t \in [k-4]$ there exists a path
  $xu^s u^t y$ where $u^s$ and $u^t$ are vertices of $U$
  assigned the colours $s$ and $t$, respectively. This implies $G[U] =
  K_{k-4}$. No vertex of $G_x$ has more than two edges missing in
  $G_x$ and so, in particular, each of the adjacent vertices $a_1$ and
  $a_2$ are adjacent to every vertex of $U$. Now $G[U \cup \{ a_1,
  a_2, x \}] = K_{k-1}$, which contradicts
  Proposition~\ref{prop:forbiddenCompleteKminusOne}. Thus, no two
  vertices of degree $k+1$ are adjacent in $G$.
\end{proof}
\section{Decomposable graphs and the ratio of double-critical edges in graphs} \label{sec:decomposable}
A graph $G$ is called \emph{decomposable} if it consists of two
disjoint non-empty subgraphs $G_{1}$ and $G_{2}$ together with all
edges joining a vertex of $G_{1}$ and a vertex of $G_{2}$.
\begin{proposition}
Let $G$ be a graph decomposable into $G_{1}$ and $G_{2}$. Then $G$
is double-critical if and only if $G_{1}$ and $G_{2}$ are both
double-critical.
\label{prop:decomposable}
\end{proposition}
\begin{proof}
Let $G$ be double-critical. Then $\chi(G) = \chi(G_1) + \chi(G_2)$.
Moreover, for $xy \in E(G_1)$ we have $$\chi(G)-2 = \chi(G-x-y) =
\chi(G_1-x-y) + \chi(G_2)$$ which implies $\chi(G_1-x-y) =
\chi(G_1)-2$. Hence $G_1$ is double-critical, and similarly $G_2$ is.

Conversely, assume that $G_{1}$ and $G_{2}$ are both
double-critical. Then for $xy \in E(G_1)$ we have $$\chi(G-x-y) =
\chi(G_1-x-y) + \chi(G_2) = \chi(G_1)-2 + \chi(G_2) = \chi(G)-2$$
For $xy \in E(G_2)$ we have similarly that $\chi(G-x-y) =
 \chi(G)-2$. For $x \in V(G_1)$ and $y \in V(G_2)$ we have $$\chi(G-x-y) =
\chi(G_1-x) + \chi(G_2-y) = \chi(G_1)-1 + \chi(G_2)-1 = \chi(G)-2$$
Hence G is double-critical.
\end{proof}

Gallai proved the theorem that a $k$-critical graph with at most
$2k-2$ vertices is always decomposable~\cite{Gallai63}. It follows
easily from Gallai's Theorem, Proposition~\ref{prop:decomposable} and
the fact that no double-critical non-complete graph with $\chi \leq 5$
exist, that a double-critical $6$-chromatic graph $G \neq K_6$ has at
least $11$ vertices. In fact, such a graph must have at least $12$ vertices. (Suppose $|V(G)| = 11$. Then $G$ cannot be
decomposable by Proposition~\ref{prop:decomposable}; moreover, no
vertex of a $k$-critical graph can have a vertex of degree $|V(G)| -
2$; hence $\Delta (G) = 8$ by
Theorem~\ref{th:noAdjacentLowVertices}, say $\deg (x) = 8$. Let $y$ and $z$ denote the two vertices of $G - N[x]$. The vertices $y$ and $z$ have to be adjacent. Hence $\chi (G - y - z) = 4$ and $\chi (G_x) = 3$, which implies $\chi (G) = 5$, a contradiction.)

It also follows from Gallai's theorem and our
results on double-critical $6$- and $7$-chromatic graphs that any
double-critical $8$-chromatic graph without $K_8$ as a minor, if it
exists, must have at least 15 vertices.

In the second part of the proof of
Proposition~\ref{prop:decomposable}, to prove that an edge $xy$ with
$x \in V(G_1)$ and $y \in V(G_2)$ is double-critical in $G$, we only
need that $x$ is critical in $G_1$ and $y$ is critical in $G_2$. Hence
it is easy to find examples of critical graphs with many
double-critical edges. Take for example two disjoint odd cycles of
equal length $\geq 5$ and join them completely by edges. The result is
a family of $6$-critical graphs in which the proportion of
double-critical edges is as high as we want, say more than 99.99
percent of all edges may be double-critical. In general, for any
integer $k \geq 6$, let $H_{k, \ell}$ denote the graph constructed by
taking the complete $(k-6)$-graph and two copies of an odd cycle
$C_\ell$ with $\ell \geq 5$ and joining these three graphs completely.
Then the non-complete graph $H_{k, \ell}$ is $k$-critical, and the
ratio of double-critical edges to the size of $H_{k, \ell}$ can be
made arbitrarily close to $1$ by choosing the integer $\ell$
sufficiently large. These observations perhaps indicate the difficulty
in proving the Double-Critical Graph Conjecture: it is not enough to
use just a few double-critical edges in a proof of the conjecture.

Taking an odd cycle $C_\ell$ ($\ell \geq 5$)and the complete $2$-graph
and joining them completely, we obtain a non-complete $5$-critical
graph with at least $2/3$ of all edges being double-critical. Maybe
these graphs are best possible:
\begin{conjecture}
  If $G$ denotes a $5$-critical non-complete graph, then $G$ contains
  at most $c := (2 + \frac{1}{3n(G) - 5}) \frac{m(G)}{3}$
  double-critical edges. Moreover, $G$ contains precisely $c$
  double-critical edges if and only if $G$ is decomposable into two
  graphs $G_1$ and $G_2$, where $G_1$ is the complete $2$-graph and
  $G_2$ is an odd cycle of length $\geq 5$.
\label{conj:5crit}
\end{conjecture}
The conjecture, if true, would be an interesting extension of the
Theorem of Mozhan~\cite{0688.05026} and Stiebitz~\cite{MR882614} that
there exists at least one non-double-critical edge. Computer tests
using the list of vertex-critical graphs made available by Royle
\cite{RoyleCriticalGraphs08} indicate that Conjecture~\ref{conj:5crit}
holds for graphs of order less than 12.  Moreover, the analogous
statement holds for $4$-critical graphs, cf.\ Theorem~\ref{th:4crit} below.
In the proof of Theorem~\ref{th:4crit} we apply the following
lemma, which is of interest in its own right.
\begin{lemma}
  No non-complete $4$-critical graph contains two non-incident
  double-critical edges.
\label{lemma:nonincidentDCedgesIn4crit}
\end{lemma}
\begin{proof}[Proof of Lemma~\ref{lemma:nonincidentDCedgesIn4crit}.]
  Suppose $G$ contains two non-incident double-critical edges $xy$ and
  $vw$. Since $ \chi ( G - \{ v,w,x,y \} ) = 2$, each component of $G
  - \{ v,w,x,y \}$ is a bipartite graph. Let $A_i$ and $B_i$ ($i \in
  [j]$) denote the partition sets of each bipartite component of $G -
  \{ v,w,x,y \}$. (For each $i \in [j]$, at least one of the sets
  $A_i$ and $B_i$ are non-empty.) Since $G$ is critical, it follows
  that no clique of $G$ is a cut set of
  $G$~\cite[Th. 14.7]{MR2368647}, in particular, both $G - x - y$ and
  $G - v - w$ are connected graphs.  Hence, in $G - v - w$, there is
  at least one edge between a vertex of $\{ x, y \}$ and a vertex of
  $A_i \cup B_i$ for each $i \in [j]$.  Similarly, for $v$ and $w$ in
  $G - x - y$. If, say $x$ is adjacent to a vertex $a_1 \in A_i$, then
  $y$ cannot be adjacent to a vertex $a_2 \in A_i$, since then there
  would be a an even length $(a_1,a_2)$-path $P$ in the induced graph
  $G[A_i \cup B_i]$ and so the induced graph $G[V(P) \cup \{x,y \}]$
  would contain an odd cycle, which contradicts the fact that the
  supergraph $G - v - w$ of $G[V(P) \cup \{x,y \}]$ is
  bipartite. Similarly, if $x$ is adjacent to a vertex of $A_i$, then
  $x$ cannot be adjacent to a vertex of $B_i$. Similar observations
  hold for $v$ and $w$. Let $A := A_1 \cup \cdots \cup A_j$ and $B :=
  B_1 \cup \cdots \cup B_j$. We may w.l.o.g. assume that the
  neighbours of $x$ in $G - v - w - y$ are in the set $A$ and the
  neighbours of $y$ in $G - v -w - x$ are in $B$.  In the following we
  distinguish between two cases.
\begin{itemize}
\item[(i)] First, suppose that, in $G - x - y$, one of the vertices
  $v$ and $w$ is adjacent to only vertices of $A \cup \{ v, w \}$,
  while the other is adjacent to only vertices of $B \cup \{ v, w \}$.
  By symmetry, we may assume that $v$ in $G - x - y$ is adjacent to
  only vertices of $A \cup \{ w \}$, while $w$ in $G - x - y$ is
  adjacent to only vertices of $B \cup \{ v \}$. In this case we
  assign the colour $1$ to the vertices of $A \cup \{ w \}$, the
  colour $2$ to the vertices of $B \cup \{ v \}$.

  Suppose that one of the edges $xv$ or $yw$ is not in $G$. By
  symmetry, it suffices to consider the case that $xv$ is not in $G$.
  In this case we assign the colour $2$ to the vertex $x$ and the
  colour $3$ to $y$. Since $x$ is not adjacent to any vertices of $B_1
  \cup \cdots \cup B_j$, we obtain a $3$-colouring of $G$, which
  contradicts the assumption that $G$ is $4$-chromatic.

  Thus, both of the edges $xv$ and $yw$ are present in $G$. Suppose
  that $xw$ or $yv$ are missing from $G$. Again, by symmetry, it
  suffices to consider the case where $yv$ is missing from $G$. Now
  assign the colour $2$ to the vertex $x$ and the colour $3$ to the vertex $y$ and a new colour to the vertex $v$. Again, we have a $3$-colouring of $G$, a
  contradiction. Thus each of the edges $xw$ and $yv$ are in $G$, and
  so the vertices $x,y,v$ and $w$ induce a complete $4$-graph in $G$.
  However, no $4$-critical graph $\neq K_4$ contains $K_4$ as a
  subgraph, and so we have a contradiction.

\item[(ii)] Suppose (i) is not the case. Then we may choose the
  notation such that there exist some integer $\ell \in \{2, \ldots, j
  \}$ such that for every integer $s \in \{1, \ldots, \ell \}$ the
  vertex $v$ is not adjacent to a vertex of $B_s$ and the vertex $w$
  is not adjacent to a vertex of $A_s$; and for every integer $t \in
  \{ \ell, \ldots, j \}$ the vertex $v$ is not adjacent to a vertex of
  $A_t$ and the vertex $w$ is not adjacent to a vertex of $B_t$.

  Since $G \nsubseteq K_4$, we may by symmetry assume that $xv \notin
  E(G)$. Now colour the vertices $v,x$ and all vertices of $B_s$
  ($s=1, \ldots, \ell - 1$) with colour $1$; colour the vertex $w$,
  all vertices of $A_s$ $(s=1, \ldots, \ell - 1$) and all vertices of
  $B_t$ ($t=\ell, \ldots, j$) with colour $2$; and colour the vertex
  $y$ and all the vertices of $A_t$ ($t=\ell, \ldots, j$) with colour
  $3$. The result is a $3$-colouring of $G$. This contradicts $G$
  being $4$-chromatic. Hence $G$ does not contain two non-incident
  double-critical edges.

\end{itemize}
\end{proof}

\begin{theorem}
  If $G$ denotes a $4$-critical non-complete graph, then $G$ contains at
  most $m(G)/2$ double-critical edges. Moreover, $G$ contains
  precisely $m(G)/2$ double-critical edges if and only if $G$ contains
  a vertex $v$ of degree $n(G) - 1$ such that the graph $G - v$ is an
  odd cycle of length $\geq 5$.
\label{th:4crit}
\end{theorem}
\begin{proof}
  Let $G$ denote a $4$-critical non-complete graph. According to
  Lemma~\ref{lemma:nonincidentDCedgesIn4crit}, $G$ contains no two
  non-incident double-critical edges, that is, every two
  double-critical edges of $G$ are incident. Then, either the
  double-critical edges of $G$ all share a common end-vertex or they
  induce a triangle. In the later case $G$ contains strictly less that
  $m(G)/2$ double-critical edges, since $n(G) \geq 5$ and $m(G) \geq
  3n(G)/2 > 6$. In the former case, let $v$ denote the common
  endvertex of the double-critical edges.

  Now, the number of double-critical edges is at most $\deg (v)$,
  which is at most $n(G) - 1$. Since $G$ is $4$-critical, it follows
  that $G - v$ is connected and $3$-chromatic. Hence $G-v$ is
  connected and contains an odd cycle, which implies $m(G-v) \geq
  n(G-v)$. Hence $m(G) = \deg (v) + m(G-v) \geq \deg (v) + n(G) - 1
  \geq 2 \deg (v)$, which implies the desired inequality. If the
  inequality is, in fact, an equality, then $\deg (v) = n(G)-1$ and
  $G$ is decomposable with $G-v$ an odd cycle of length $\geq 5$. The
  reverse implication is just a simple calculation. The reverse
  implication is just a simple calculation.
\end{proof}

\section{Connectivity of double-critical graphs}
\label{sec:connectivity}
\begin{proposition}
  Suppose $G$ is a non-complete double-critical $k$-chromatic graph
  with $k \geq 6$. Then no minimal separating set of $G$ can be
  partitioned into two disjoint sets $A$ and $B$ such that the induced
  graphs $G[A]$ and $G[B]$ are edge-empty and complete, respectively.
\label{prop:miniSepSetNoPartition}
\end{proposition}
\begin{proof}
  Suppose that some minimal separating set $S$ of $G$ can be
  partitioned into disjoint sets $A$ and $B$ such that $G[A]$ and
  $G[B]$ are edge-empty and complete, respectively. We may assume that
  $A$ is non-empty. Let $H_1$ denote a component of $G - S$, and let
  $H_2 := G - (S \cup V(H_1))$. Since $A$ is not empty, there is at
  least one vertex $x \in A$, and, by the minimality of the separating
  set $S$, this vertex $x$ has neighbours in both $V(H_1)$ and
  $V(H_2)$, say $x$ is adjacent to $y_1 \in V(H_1)$ and $y_2 \in
  V(H_2)$. Since $G$ is double-critical, the graph $G - x - y_2$ is
  $(k-2)$-colourable, in particular, there exists a $(k-2)$-colouring
  $\varphi_1$ of the subgraph $G_1 := G[V(H_1) \cup B]$. Similarly,
  there exists a $(k-2)$-colouring $\varphi_2$ of $G_2 := G[V(H_2)
  \cup B]$. The two graphs have precisely the vertices of $B$ in
  common, and the vertices of $B$ induce a complete graph in both
  $G_1$ and $G_2$. Thus, both $\varphi_1$ and $\varphi_2$ use exactly
  $|B|$ colours to colour the vertices of $B$, assigning each vertex a
  unique colour. By permuting the colours assigned by, say
  $\varphi_2$, to the vertices of $B$, we may assume $\varphi_1 (b) =
  \varphi_2(b)$ for every vertex $b \in B$. Now $\varphi_1$ and
  $\varphi_2$ can be combined into a $(k-2)$-colouring $\varphi$ of $G
  - A$.  This colouring $\varphi$ may be extended to a
  $(k-1)$-colouring of $G$ by assigning every vertex of the
  independent set $A$ the some new colour. This contradicts the fact
  that $G$ is $k$-chromatic, and so no minimal separating set $S$ as
  assumed can exist.
\end{proof}
Krusenstjerna-Hafstr{\o}m and Toft~\cite{MR634546} states that any
double-critical $k$-chromatic ($k \geq 5$) non-complete graph is
$5$-connected and $k+1$-edge-connected. In the following we prove that
any double-critical $k$-chromatic ($k \geq 6$) non-complete graph is
$6$-connected.
\begin{theorem}
  Every double-critical $k$-chromatic non-complete graph with $k \geq
  6$ is $6$-connected.
\label{th:doubleCriticalImpliesSixConnected}
\end{theorem}
\begin{proof}
  Recall, that any double-critical graph, by definition, is connected.
  Thus, since $G$ is not complete, there exists some subset $U
  \subseteq V(G)$ such that $G - U$ is disconnected. Let $S$ denote a
  minimal separating set of $G$. We show $|S| \geq 6$. If
  $|S| \leq 3$, then $S$ can be partitioned into two disjoint subset
  $A$ and $B$ such that the induced graphs $G[A]$ and $G[B]$ are
  edge-empty and complete, respectively, and, thus, we have a
  contradiction by Proposition~\ref{prop:miniSepSetNoPartition}.
  Suppose $|S| \geq 4$, and let $H_1$ and $H_2$ denote disjoint
  non-empty subgraphs of $G - S$ such that $G-S = H_1 \cup H_2$.

  If $|S| \leq 5$, then each vertex $v$ of $V(H_1)$ has at most five
  neighbours in $S$ and so $v$ must have at least two neighbours in
  $V(H_1)$, since $\delta (G) \geq k+1 \geq 7$. In particular, there
  is at least one edge $u_1 u_2$ in $H_1$, and so $G - u_1 - u_2$ is
  $(k-2)$-colourable. This implies that the subgraph $G_2 := G - H_1$
  of $G - u_1 - u_2$ is $(k-2)$-colourable. Let $\varphi_2$ denote a
  $(k-2)$-colouring of $G_2$. A similar argument shows that $G_1 := G
  - H_2$ is $(k-2)$-colourable. Let $\varphi_1$ denote a $(k-2)$-colouring of
  $G_1$. If $\varphi_1$ or $\varphi_2$ applies just one colour to the vertices of
  $S$, then $S$ is an independent set of $G$, which contradicts
  Proposition~\ref{prop:miniSepSetNoPartition}.  Thus, we may assume
  that both $\varphi_1$ and $\varphi_2$ applies at least two colours to the
  vertices of $S$. Let $|\varphi_i(S)|$ denote the number of colours applied
  by $\varphi_i$ ($i=1,2$) to the vertices of $S$. By symmetry, we may
  assume $|\varphi_1(S)| \geq |\varphi_2(S)| \geq 2$.

  Moreover, if $|\varphi_1 (S)| = |\varphi_2(S)| = |S|$, then,
  clearly, the colours applied by say $\varphi_1$ may be permuted such
  that $\varphi_1(s) = \varphi_2(s)$ for every $s \in S$ and so
  $\varphi_1$ and $\varphi_2$ may be combined into a $(k-2)$-coloring
  of $G$, a contradiction. Thus, $|\varphi_1 (S)| = |S|$ implies
  $|\varphi_2(S)| < |S|$.

  In general, we redefine the $(k-2)$-colourings $\varphi_1$ and
  $\varphi_2$ into $(k-1)$-colourings of $G_1$ and $G_2$,
  respectively, such that, after a suitable permutation of the colours
  of say $\varphi_1$, $\varphi_1(s) = \varphi_2(s)$ for every vertex
  $s \in S$. Hereafter a proper $(k-1)$-colouring of $G$ may be
  defined as $\varphi(v) = \varphi_1(v)$ for every $v \in V(G_1)$ and
  $\varphi(v) = \varphi_2 (v)$ for every $v \in V(G) \backslash
  V(G_1)$, which contradicts the fact that $G$ is $k$-chromatic. In
  the following cases we only state the appropriate redefinition of
  $\varphi_1$ and $\varphi_2$.

  Suppose that $|S| = 4$, say $S = \{ v_1, v_2, v_3, v_4 \}$.  We
  consider several cases depending on the values of
  $|\varphi_1(S)|$ and $|\varphi_2(S)|$. If $|\varphi_i(S)| = 2$ for some $i
  \in \{ 1,2 \}$, then $\varphi_i$ must apply both colours twice on vertices
  of $S$ (by Proposition~\ref{prop:miniSepSetNoPartition}).
\begin{itemize}
\item[1)] Suppose that $|\varphi_1(S)| = 4$.
\item[1.1)] Suppose that $|\varphi_2(S)| = 3$. In this case
  $\varphi_2$ uses the same colour at two vertices of $S$, say
  $\varphi_2(v_1) = \varphi_2(v_2)$. We simply redefine $\varphi_2$
  such that $\varphi_2(v_1) = k-1$. Now both $\varphi_1$ and
  $\varphi_2$ applies four distinct colours to the vertices of $S$ and
  so they may be combined into a $(k-1)$-colouring of $G$, a
  contradiction.
\item[1.2)] Suppose that $|\varphi_2(S)| = 2$, say $\varphi_2 (v_1) =
  \varphi_2 (v_2)$ and $\varphi_2(v_3) = \varphi_2(v_4)$. This, in
  particular, implies $v_1v_2 \notin E(G)$, and so $\varphi_1$ may be
  redefined such that $\varphi_1 (v_1) = \varphi_1 (v_2) = k-1$.
  Moreover, $\varphi_2$ is redefined such that $\varphi_2(v_4) = k-1$.
\item[2)] Suppose that $|\varphi_1(S)| = 3$, say $\varphi_1 (v_1) =
  1$, $\varphi_1 (v_2) = 2$ and $\varphi_1(v_3) = \varphi_1(v_4) = 3$.
\item[2.1)] Suppose that $|\varphi_2 (S)| = 3$, say $\varphi_2 (x) =
  \varphi_2 (y)$ for two distinct vertices $x,y \in S$. Redefine
  $\varphi_1$ and $\varphi_2$ such that $\varphi_1 (v_4) = k-1$ and
  $\varphi_2(x) = k-1$.
\item[2.2)] Suppose that $|\varphi_2 (S)| = 2$. If $\varphi_2(v_1) =
  \varphi_2(v_2)$ and $\varphi_2 (v_3) = \varphi_2 (v_4)$, then the
  desired $(k-1)$-colourings are obtained by redefining $\varphi_2$
  such that $\varphi_2 (v_2) = k-1$. If $\varphi_2(v_2) =
  \varphi_2(v_3)$ and $\varphi_2 (v_4) = \varphi_2 (v_1)$, then the
  desired $(k-1)$-colourings are obtained by redefining $\varphi_2$
  such that $\varphi_2 (v_3) = \varphi_2 (v_4 ) = k-1$.
\item[3)] Suppose that $|\varphi_1(S)|=2$, which implies
  $|\varphi_2(S)|=2$. We may assume $\varphi_1 (v_1) = \varphi_1
  (v_2)$ and $\varphi_1 (v_3) = \varphi_1 (v_4)$, in particular, $v_1
  v_2 \notin E(G)$. If $\varphi_2 (v_1) = \varphi_2 (v_2)$ and
  $\varphi_2 (v_3) = \varphi_2 (v_4)$, then, obviously, $\varphi_1$
  and $\varphi_2$ may be combined into a $(k-2)$-colouring of $G$, a
  contradiction. Thus, we may assume that $\varphi_2 (v_2) =
  \varphi_2(v_3)$ and $\varphi_2(v_4)= \varphi_2 (v_1)$. In this case
  we redefine both $\varphi_1$ and $\varphi_2$ such that $\varphi_1
  (v_4) = k-1$, and, since $v_1 v_2 \notin E(G)$, $\varphi_2(v_1) =
  \varphi_2 (v_2) = k-1$.
\end{itemize}
This completes the case $|S|=4$. Suppose $|S| = 5$, say $S = \{ v_1,
v_2, v_3, v_4, v_5 \}$. According to
Proposition~\ref{prop:miniSepSetNoPartition}, neither $\varphi_1$ nor
$\varphi_2$ uses the same colour for more than three vertices. Suppose
that one of the colourings $\varphi_1$ or $\varphi_2$, say
$\varphi_2$, applies the same colour to three vertices of $S$, say
$\varphi_2 (v_3) = \varphi_2 (v_4) = \varphi_2 (v_5)$. Now $\{ v_3,
v_4, v_5 \}$ is an independent set.  If (i) $\varphi_1 (v_1) =
\varphi_1(v_2)$ and $\varphi_2 (v_1) = \varphi_2(v_2)$ or (ii)
$\varphi_1 (v_1) \neq \varphi_1(v_2)$ and $\varphi_2 (v_1) \neq
\varphi_2(v_2)$, then we redefine $\varphi_1$ such that $\varphi_1
(v_3) = \varphi_1 (v_4) = \varphi_1 (v_5) = k-1$, and so $\varphi_1$
and $\varphi_2$ may, after a suitable permutation of the colours of
say $\varphi_1$, be combined into a $(k-1)$-colouring of $G$.
Otherwise, if $\varphi_1 (v_1) \neq \varphi_1(v_2)$ and $\varphi_2
(v_1) = \varphi_2 (v_2)$, then we redefine both $\varphi_1$ and
$\varphi_2$ such that $\varphi_1 (v_3) = \varphi_1 (v_4) = \varphi_1
(v_5) = k-1$ and $\varphi_2 (v_2) = k-1$. If $\varphi_1 (v_1) =
\varphi_1(v_2)$ and $\varphi_2 (v_1) \neq \varphi_2 (v_2)$, then we
redefine both $\varphi_1$ and $\varphi_2$ such that $\varphi_1 (v_3) =
\varphi_1 (v_4) = \varphi_1 (v_5) = k-1$ and $\varphi_2 (v_1) =
\varphi_2 (v_2) = k-1$. In both cases $\varphi_1$ and $\varphi_2$ may
be combined into a $(k-1)$-colouring of $G$ Thus, we may assume that
neither $\varphi_1$ nor $\varphi_2$ applies the same colour to three
or more vertices of $S$, in particular, $|\varphi_i (S)| \geq 3$ for both $i
\in \{ 1,2 \}$. Again, we may assume $|\varphi_1 (S)| \geq
|\varphi_2(S)|$.

\begin{itemize}
\item[a)] Suppose that $|\varphi_1 (S)| = 5$.
\item[a.1)] Suppose that $|\varphi_2(S)| = 4$ with say $\varphi_2
  (v_4) = \varphi_2 (v_5)$. In this case $v_4v_5 \notin E(G)$ and so
  we redefine $\varphi_1$ such that $\varphi_1 (v_4) = \varphi_1 (v_5)
  = k-1$.
\item[a.2)] Suppose that $|\varphi_2 (S)| = 3$. Since $\varphi_2$
  cannot assign the same colour to three or more vertices of $S$, we
  may assume $\varphi_2 (v_2) = \varphi_2 (v_3)$ and $\varphi_2(v_4) =
  \varphi_2(v_5)$. In this case $v_4v_5 \notin E(G)$, and so we
  redefine $\varphi_1$ and $\varphi_2$ such that $\varphi_1 (v_4) =
  \varphi_1 (v_5) = k-1$ and $\varphi_2 (v_3) = k-1$.
\item[b)] Suppose $|\varphi_1 (S)| = 4$, say $\varphi_1(v_4) = \varphi_1 (v_5)$.
\item[b.1)] Suppose $|\varphi_2(S)| = 4$ with $\varphi_2(x) =
  \varphi_2 (y)$ for two distinct vertices $x,y \in S$. In this case
  we redefine $\varphi_1$ and $\varphi_2$ such that $\varphi_1 (v_5 )
  = k-1$ and $\varphi_2(y) = k-1$.
\item[b.2)] Suppose $|\varphi_2 (S)| = 3$. In this case we distinguish
  between two subcases depending on the number of colours $\varphi_2$
  applies to the vertices of the set $\{ v_1, v_2, v_3 \}$. As noted
  earlier, we must have $|\varphi_2 ( \{ v_1, v_2, v_3 \} ) | \geq 2$.
  If $|\varphi_2 ( \{ v_1, v_2, v_3 \} ) | = 3$, then we redefine
  $\varphi_2$ such that $\varphi_2 (v_4) = \varphi_2 (v_5) = k-1$.
  Otherwise, if $|\varphi_2 ( \{ v_1, v_2, v_3 \} ) | = 2$ with say
  $\varphi_2(v_2) = \varphi_2(v_3)$. Now $v_2 v_3, v_4 v_5 \notin
  E(G)$ and so we redefine $\varphi_1$ and $\varphi_2$ such that
  $\varphi_1 (v_2) = \varphi_1 (v_3) = k-1$ and $\varphi_2(v_4) =
  \varphi_2 (v_5) = k-1$.
\item[c)] Suppose that $|\varphi_1 (S)|=3$, say $\varphi_1 (v_2) =
  \varphi_1 (v_3)$ and $\varphi_1(v_4) = \varphi_1 (v_5)$.  In this
  case we must have $|\varphi_2(S)|=3$. As noted earlier, $\varphi_2$
  does not assign the same colour to three vertices of $S$, and so we
  may assume $\varphi_2$ applies the colours $1,2$ and $3$ to the
  vertices of $S$ and that only one vertex of $S$ is assigned the
  colour $1$ while two pairs of vertices of given the colours $2$ and
  $3$, respectively.  We distinguish between four subcases depending
  on which vertex of $S$ is assigned the colour $1$ by $\varphi_2$ and
  and the number of colours $\varphi_2$ applies to the vertices of the
  two sets $\{ v_2, v_3 \}$ and $\{ v_4, v_5 \}$. We may assume
  $|\varphi_2 ( \{ v_2, v_3 \})| \geq |\varphi_2( \{ v_4, v_5 \} )|$.
\item[c.1)] If $|\varphi_2 ( \{ v_2, v_3 \})| = |\varphi_2( \{ v_4,
  v_5 \} )|= 1$, then, clearly, $\varphi_1$ and $\varphi_2$ may be
  combined into a $(k-2)$-colouring of $G$, a contradiction.
\item[c.2)] Suppose $|\varphi_2 ( \{ v_2, v_3 \})| = 2$, $|\varphi_2(
  \{ v_4, v_5 \} )| = 2$ and $\varphi_2 (v_1) = 1$. Suppose that
  $\varphi_2$ assigns the colour $2$ to the two distinct vertices $x,
  y \in S \backslash \{ v_1 \}$. Now we redefine $\varphi_1$ and
  $\varphi_2$ such that $\varphi_1(x) = \varphi_1 (y) = k-1$ and
  $\varphi_2 (z) = k-1$ for some vertex $z \in S \backslash \{ v_1, x,
  y \}$.
\item[c.3)] Suppose $|\varphi_2 ( \{ v_2, v_3 \})| = 2$, $|\varphi_2(
  \{ v_4, v_5 \} )| = 2$ and $\varphi_2 (v_1) \neq 1$, say $\varphi_2
  (v_5) = 1$. In this case there is a vertex $x \in \{ v_2, v_3 \}$
  such that $\varphi_2 (x) = \varphi_2 (v_4)$. Now we redefine
  $\varphi_1$ and $\varphi_2$ such that $\varphi_1 (x) = \varphi_1
  (v_4) = k-1$ and $\varphi_2 (v_1) = k-1$.
\item[c.4)] If $|\varphi_2 ( \{ v_2, v_3 \})| = 2$, $|\varphi_2( \{
  v_4, v_5 \} )| = 1$, then we redefine $\varphi_2$ such that
  $\varphi_2 (v_2) = \varphi_2 (v_3) = k-1$.
\end{itemize}
\end{proof}
\section{Double-critical $6$-chromatic graphs}
\label{sec:dc6chrom}
In this section we prove, without use of the Four Colour Theorem, that
any double-critical $6$-chromatic graph is contractible to $K_6$.
\begin{theorem}
  Every double-critical $6$-chromatic graph $G$ contains $K_6$ as a minor.
\label{th:DoubleCriticalChrom6}
\end{theorem}
\begin{proof}
  If $G$ is a the complete $6$-graph, then we are done. Hence we may
  assume that $G$ is not the complete $6$-graph. Now, according to
  Proposition~\ref{prop:minimumDegree}, $\delta (G) \geq 7$. Firstly,
  suppose that $\delta (G) \geq 8$. Then $m(G) = \frac{1}{2} \sum_{v
    \in V(G)} \deg (v) \geq 4n(G) > 4n(G) - 9$.
  Gy{\H{o}}ri~\cite{MR652887} and Mader~\cite{MR0229550} proved that
  any graph $H$ with $m(H) \geq 4n(G) - 9$ is contractible to $K_6$,
  which implies the desired result. Secondly, suppose that $G$
  contains a vertex, say $x$, of degree $7$. Let $y_i$ $(i \in [7])$
  denote the neighbours of $x$.  Now, according to
  Proposition~\ref{prop:structureOfGx}, the complement of the induced
  subgraph $G_x$ consists of isolated vertices and cycles (at least
  one) of length at least five. Since $n(G_x) = 7$, the complement
  $\overline{G_x}$ must contain exactly one cycle $C_\ell$.  We
  consider three cases depending on the length of $C_\ell$.
\begin{itemize}
\item[(i)] Suppose $\ell = 5$, say $C_\ell = \{ y_1, y_2, y_3, y_4,
  y_5 \}$. Now $\{y_1, y_3, y_6, y_7 \}$ induces a $K_4$, and so
  $\{y_1, y_3, y_6, y_7, x \}$ induces a $K_5$, which contradicts
  Proposition~\ref{prop:forbiddenCompleteKminusOne}.
\item[(ii)] Suppose $\ell = 6$, say $C_\ell = \{ y_1, y_2, y_3, y_4,
  y_5, y_6 \}$. In this case, $\{y_1, y_3, y_5, y_7, x \}$ induces a
  $K_5$, again, we obtain a contradiction.
\item[(iii)] Finally, $\ell = 7$, say $C_\ell = \{ y_1, y_2, y_3, y_4,
  y_5, y_6, y_7 \}$. Now by contracting the edges $y_2y_5$ and
  $y_4y_7$ of $G_x$ into two distinct vertices a complete $5$-graph is
  obtained, as is readily verified. Since, by definition, the vertex
  $x$ is adjacent to every vertex of $V(G_x)$, it follows that $G$ is
  contractible to $K_6$.
\end{itemize}
\end{proof}
The proof of Theorem~\ref{th:DoubleCriticalChrom6} implies the
following result.
\begin{corollary}
  Every double-critical $6$-chromatic graph $G$ with $\delta (G) = 7$
  has the property that for every vertex $x \in V(G)$ with $\deg (x) =
  7$, the complement $\overline{G_x}$ is a $7$-cycle.
\label{cor:hamiltonianComplement}
\end{corollary}
\section{Double-critical $7$-chromatic  graphs}
\label{sec:dc7chrom}
Let $G$ denote a double-critical non-complete $7$-chromatic
graph. Recall, that given a vertex $x \in V(G)$, we let $G_x$ denote
the induced graph $G[N(x)]$ and $\alpha_x := \alpha (G_x)$. The
following corollary is a direct consequence of
Proposition~\ref{prop:kMinusTwoColourableNeighbourGraphGx}.
\begin{corollary}
  For any vertex $x$ of $G$ not joined to all other vertices, $\chi (G_x) \leq
  4$.
\label{cor:FourColourableNeighbourGraphGx}
\end{corollary}
\begin{proposition}
  For any vertex $x$ of $G$ of degree $9$, $\alpha_x = 3$.
\label{prop:DoubleCriticalContractionCriticalChrom7Alpha3}
\end{proposition}
\begin{proof} It follows from
  Proposition~\ref{prop:upperBoundOnAlphaGx}, that $\alpha_x$ is at
  most $3$. Since $\chi(G_x) \cdot \alpha_x \geq n(G_x) = 9$, it
  follows from Corollary~\ref{cor:FourColourableNeighbourGraphGx},
  that $\alpha_x \geq 9/\chi(G_x) \geq 9/4$, which implies $\alpha_x
  \geq 3$. Thus, $\alpha_x = 3$.
\end{proof}
\begin{proposition}
  If $x$ is a vertex of degree $9$ in $G$, then the complement
  $\overline{G_x}$ does not contain a $K_4^-$ as a subgraph.
\label{th:ForbiddenK4missingOneEdge}
\end{proposition}
\begin{proof}
  Let $x$ denote a vertex of degree $9$ in $G$. By
  Proposition~\ref{prop:minimumDegreeInGx}, the minimum degree in
  $G_x$ is at least $k-2 = 5$. Suppose that the vertices $y_1, y_2,
  z_1, z_2$ are the vertices of a subgraph $K_4^-$ in
  $\overline{G_x}$, that is, a $4$-cycle with a diagonal edge
  $y_1y_2$. The graph $G - x - y_1$ is $5$-colourable, and, according
  to Corollary~\ref{cor:atLeastKminusTwoCommonNeighbours}, every one
  of the five colours occurs in $B(xy_1)$. None of the vertices $y_2,
  z_1$ or $z_2$ are in $B(xy_1)$, that is, $B(xy_1) \subseteq V(G_x)
  \backslash \{y_1,y_2,z_1,z_2 \}$. Now the vertex $y_2$ is not
  adjacent to every vertex of $B(xy_1)$, since that would leave none
  of the five colours available for properly colouring $y_2$. Thus, in
  $G_x$ the vertex $y_2$ has at least four non-neighbours ($y_1, z_1,
  z_2$ and, at least, one vertex from $B(xy_1)$). Since $n(G_x) = 9$,
  we find that $y_2$ has at most $8 - 4$ neighbours in $N[x]$, and we
  have a contradiction.
\end{proof}
\begin{proposition}
  For any vertex $x$ of degree $9$ in $G$, any vertex of an
  $\alpha(G_x)$-set has degree $5$ in the neighbourhood graph $G_x$.
\label{prop:degreeInAlphaXSet}
\end{proposition}
\begin{proof}
  Let $x$ denote vertex of $G$ of degree $9$, and let $W = \{w_1,
  w_2, w_3 \}$ denote any independent set in $G_x$. This vertices of
  $W$ all have degree at most $6$ in $G_x$ and, by
  Proposition~\ref{prop:minimumDegreeInGx}, at least $5$. Suppose
  that, say, $w_1 \in W$ has degree $6$. Now $B(xw_2)$ is a subset of
  $N(w_1; G_x)$, $G - x - w_2$ is $5$-colourable, and, according
  to Corollary~\ref{cor:atLeastKminusTwoCommonNeighbours}, every
  one of the five colours occurs in $B(xy_1)$. This, however, leaves
  none of the five colours available for $w_1$, and we have a contradiction. It
  follows that any vertex of an independent set of three vertices in
  $G_x$ have degree $5$ in $G_x$.
\end{proof}
\begin{proposition}
  If $G$ has a vertex $x$ of degree $9$, then 
\begin{itemize}
\item[\tn{(a)}] the vertices of any maximum independent set $W = \{w_1, w_2,
  w_3 \}$ all have degree $5$ in $G_x$,
\item[\tn{(b)}] the vertices of $V(G_x)$ have degree $5$, $6$ or $8$ in $G_x$,
\item[\tn{(c)}] every vertex $w_i$ ($i=1,2,3$) has exactly one private
  non-neighbour w.r.t.\ $W$ in $G_x$, that is, there
  exist three distinct vertices in $G_x - W$, which we denote by
  $y_1$, $y_2$ and $y_3$, such that each $w_i$ ($i=1,2,3$) is adjacent
  to every vertex of $G_x - (W \cup y_i)$, and
\item[\tn{(d)}] each vertex $y_i$ has a neighbour and non-neighbour in
  $V(G_x) \backslash ( W \cup \{y_1, y_2, y_3 \} )$ (see
  Figure~\ref{fig:baseCaseExtended}).
\end{itemize}
\begin{figure}
\begin{center}
\input{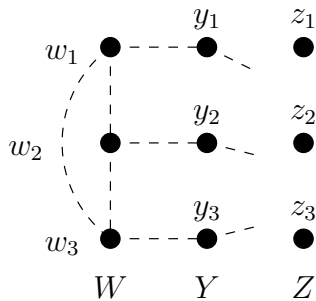}
\caption{The graph $G_x$ as described in
  Proposition~\ref{prop:baseCaseExtended}. The dashed curves indicate
  missing edges. The missing edges from $W$ to $Y \cup Z$ are exactly
  as indicated in the figure, while there may be more missing edges in
  $E(G_x - W)$ than indicated. The dashed curves starting at vertices
  of $y_i$ ($i=1,2,3$) and not ending at a vertex represent a missing
  edges between $y_i$ and a vertex of $Z$.}
\label{fig:baseCaseExtended}
\end{center}
\end{figure}
\label{prop:baseCaseExtended}
\end{proposition}
In the following, let $W := \{ w_1, w_2, w_3 \}$, $Y := \{ y_1, y_2,
y_3 \}$ and $Z := V(G_x) \backslash (W \cup Y)$.  Note that the above
corollary does not claim that each vertex $y_i$ has a private
non-neighbour in $Z$ w.r.t.\ to $Y$.
\begin{proof}
  Claim (a) follows from
  Proposition~\ref{prop:DoubleCriticalContractionCriticalChrom7Alpha3}
  and Proposition~\ref{prop:degreeInAlphaXSet}. According to
  Proposition~\ref{prop:minimumDegreeInGx}, $\delta (G_x) \geq 5$,
  and, obviously, $\Delta (G_x) \leq 8$, since $n(G_x) = 9$. If some
  vertex $y \in G_x$ has degree strictly less than $8$, then,
  according to Proposition~\ref{prop:missingNeighboursInGx}, it has at
  least two non-neighbours in $G_x$, that is, $\deg (y, G_x) \leq
  8-2$.  This establishes (b). As for the claim (c), each vertex $w_i$
  ($i=1,2,3$) has exactly five neighbours in $V(G_x) \backslash W$,
  which is a set of six vertices, and so $w_i$ has exactly one
  non-neigbour in $V(G_x) \backslash W$. Suppose say $w_1$ and $w_2$
  have a common non-neighbour in $V(G_x) \backslash W$, say $u$. Now
  the vertices $w_1, w_2, w_3$ and $u$ induce a $K_4$ or $K_4^-$ in
  the complement $\overline{G_x}$, which contradicts
  Propositions~\ref{th:ForbiddenK4missingOneEdge}. Hence, (c) follows.
  Now for claim (d). The fact that each vertex $y_i$ in $Y$ has at
  least one neighbour in $Z$ follows (b) and the fact that $y_i$ is
  not adjacent to $w_i$. It remains to show that $y_i$ has at least
  one non-neighbour in $Z$. The graph $G - x - w_1$ is $5$-colourable,
  in particular, there exists a $5$-colouring $c$ of $G_x - w_1$,
  which, according to
  Corollary~\ref{cor:atLeastKminusTwoCommonNeighbours}, assigns
  every colour of $[5]$ to at least one vertex of $B(xw_1)$. In this
  case $B(xw_1)$ consists of precisely the vertices $y_2, y_3, z_1,
  z_2$ and $z_3$. We may assume $\varphi(y_2) = 1$, $\varphi(y_3)=2$, $\varphi(z_1) =
  3$, $\varphi(z_2) = 4$ and $\varphi(z_3) = 5$. Since $w_2$ is adjacent to every
  vertex of $Z \cup Y \backslash \{ y_2 \}$, the only colour available
  for $w_2$ is the colour assign to $y_2$, that is, $\varphi(w_2) = \varphi(y_2) =
  1$. Similarly, $\varphi(w_3) = \varphi(y_3) = 2$. Both the vertices $w_2$ and
  $w_3$ are adjacent to $y_1$ and so the colour assigned to $y_1$
  cannot be one of the colours $1$ or $2$, that is, $\varphi(y_1) \in \{
  3,4,5 \}$. This implies, since $\varphi(z_1) = 3$, $\varphi(z_2) = 4$ and
  $\varphi(z_3) = 5$, that $y_1$ cannot be adjacent to all three vertices
  $z_1$, $z_2$ and $z_3$. Thus, (d) is established.
\end{proof}
\begin{corollary}
  If $G$ has a vertex $x$ of degree $9$, then there are at least two
  edges between vertices of $Y$.
\label{cor:AtLeastTwoEdgesInY}
\end{corollary}
\begin{proof}
  If $m(G[Y]) \leq 1$, then it follows from (c) and (d) of
  Proposition~\ref{prop:baseCaseExtended}, that some vertex $y_i \in
  Y$ has at most four neighbours in $G_x$. But this contradicts (b) of
  the same proposition. Thus, $m(G[Y]) \geq 2$.
\end{proof}
\begin{figure}
\begin{center}
\input{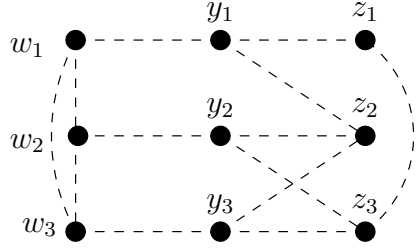}
\caption{The graph $F$. The dashed lines between vertices indicate
  missing edges. Any edge which is not explicity indicated missing is
  present in $F$.}
\label{fig:F}
\end{center}
\end{figure}
\begin{lemma}
  If $x$ is a vertex of $G$ with minimum degree $9$ and the
  neighbourhood graph $G_x$ is isomorphic to the graph $F$ of
  Figure~\ref{fig:F}, then $G$ is contractible to $K_7$.
\label{lem:F}
\end{lemma}
\begin{proof}
  According to Corollary~\ref{cor:FourColourableNeighbourGraphGx},
  $\chi (G[N[x]]) \leq 5$, and so $N[x] \neq V(G) $. Let $H$ denote
  some component in $G - N[x]$. There are several ways of contracting
  $G_x$ to $K_6^-$. For instance, by contracting the three edges $w_1
  y_3$, $w_2 y_1$ and $w_3 y_2$ into three distinct vertices a $K_6^-$
  is obtained, where the vertices $z_1$ and $z_3$ remain non-adjacent.
  Thus, if there were a $z_1$-$z_3$-path $P(z_1, z_3)$ with internal
  vertices completely contained in the set $V(G) \backslash N[x]$, then, by
  contracting the edges of $P( z_1, z_3)$, we would have a
  neighbourhood graph of $x$, which were contractible to $K_6$. Similarly,
  there exists contractions of $G_x$ such that if only there were a
  $w_1$-$y_1$-path $P(w_1, y_1)$, $w_2$-$y_2$-path $P(w_2, y_1)$ or
  $w_3$-$y_3$-path $P(w_3, y_3)$ with internal vertices completely
  contained in the set $V(G)\backslash N[x]$, then such a path could be
  contracted such that the neighbourhood graph of $x$ would be
  contractible to $K_6$. Assume that none of the above mentioned paths
  $P( z_1, z_3)$, $P(w_1, y_1)$, $P(w_2, y_1)$ and $P(w_3, y_3)$
  exist. In particular, for each pair of vertices $( z_1, z_3)$,
  $(w_1, y_1)$, $(w_2, y_2)$ and $(w_3, y_3)$ at most one vertex is
  adjacent to a vertex of $V(H)$, since if both, say $z_1$ and $z_3$
  were adjacent to, say $u \in V(H)$ and $v \in V(H)$, respectively,
  then there would be a $z_1$-$z_3$-path with internal vertices
  completely contained in the set $V(G) \backslash N[x]$,
  contradicting our assumption. Now it follows that in $G$ there can
  be at most five vertices of $V(G_x)$ adjacent to vertices of $V(H)$.
  By removing from $G$ the vertices of $V(G_x)$, which are adjacent to
  vertices of $V(H)$, the graph splits into at least two distinct
  components with $x$ in one component and the vertices of $V(H)$ in
  another component.  This contradicts
  Theorem~\ref{th:doubleCriticalImpliesSixConnected}, which states
  that $G$ is $6$-connected, and so the proof is complete.
\end{proof}
\begin{theorem}
  Every double-critical $7$-chromatic graph $G$ contains $K_7$ as a
  minor.
\label{th:doubleCriticalContractionCriticalSevenChrom}
\end{theorem}
\begin{proof}
  If $G$ is a complete $7$-graph, then we are done. Hence, we may
  assume that $G$ is not a complete $7$-graph, and so, according to
  Proposition~\ref{prop:minimumDegree}, $\delta (G) \geq 8$. If
  $\delta (G) \geq 10$, then $m(G) \geq 5n(G) > 5n - 14$, and it
  follows from a theorem of Gy{\H{o}}ri~\cite{MR652887} and
  Mader~\cite{MR0229550} that $G$ contains $K_7$ as a minor. Let $x$
  denote a vertex of minimum degree. Suppose $\delta (G) = 8$. Now,
  according to Proposition~\ref{prop:structureOfGx}, the complement
  $\overline{G_x}$ consists of isolated vertices and cycles (at least
  one), each having length at least five. Since $n(G_x) = 9$, it
  follows that $\overline{G_x}$ contains exactly one cycle $C_\ell$ of
  length $\geq 5$.
\begin{itemize}
\item[(i)]
If $\ell = 5$, then $G[y_1, y_3, y_6, y_7, y_8, x]$ is the complete
$6$-graph, a contradiction.
\item[(ii)]
If $\ell = 6$, then $G[y_1, y_3, y_5, y_7, y_8, x]$ is the complete
$6$-graph, a contradiction.
\item[(iii)]
If $\ell = 7$, then by contracting the edges $y_1 y_4$ and $y_2 y_6$
of $G_x$ into two distinct vertices a complete $6$-graph is obtained,
and so $G \geq K_7$.
\item[(iv)]
If $\ell = 8$, then by contracting the edges $y_1y_5$ and $y_3y_7$ of
$G_x$ into two distinct vertices a complete $6$-graph is obtained, and
so $G \geq K_7$.
\end{itemize}
Now, suppose $\delta (G) = 9$. According to
Proposition~\ref{prop:baseCaseExtended}, there exists an
$\alpha_x$-set $W = \{w_1, w_2, w_3 \}$ of three distinct vertices
such that there is a set $Y = \{y_1,y_2,y_3\} \subseteq V(G)
\backslash W$ of three distinct vertices such that $N(w_i, G_x) =
V(G_x) \backslash (W \cup y_i)$ (see
Figure~\ref{fig:baseCaseExtended}). Let $Z = \{ z_1, z_2, z_3\}$
denote the three remaining vertices of $G_x - (W\cup Y)$. We shall
investigate the structure of $G_x$ and consider several cases. Thus,
$e(W)=0$, and, as follows from Corollary~\ref{cor:AtLeastTwoEdgesInY},
$e(Y) \geq 2$.
\begin{figure}
\begin{center}
\input{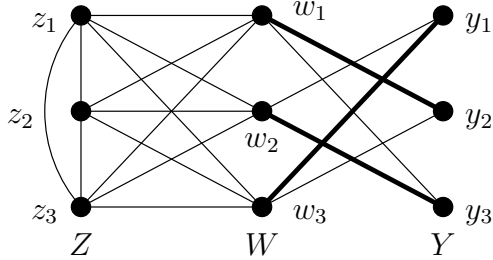}
\caption{In Case 1.2.3, the graph $G_x$ contains the graph depicted above as a
  subgraph. The thick curves indicate the edges to be contracted. By
  contracting the three edges of $G_x$ as indicated above, a $K_6$ minor
  is obtained.}
\label{fig:Zcomplete}
\end{center}
\end{figure}
Suppose $e(Z) = 3$. By contracting the edges $w_1y_2$, $w_2y_3$
and $w_3y_1$ of $G_x$ into three distinct vertices a complete
$6$-graph is obtained~(see Figure~\ref{fig:Zcomplete}). Thus, $G \geq
K_7$. In the following we shall be assuming $e(Z) \leq 2$.
\begin{figure}
\begin{center}
\input{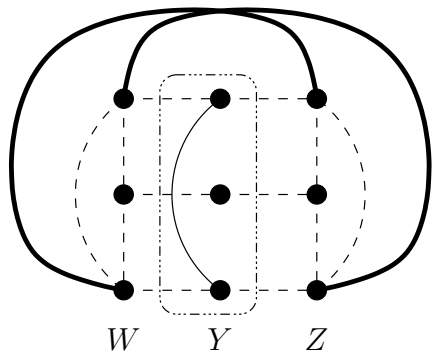}
\caption{The graph $G_x$ contains the graph depicted above as a
  subgraph. The dashed curves represent edges missing in $G_x$.
  Except for the edges of $E(Y)$, any two pair of edge which are not
  explicity shown as non-adjacent are adjacent. The edge-set $E(Y)$
  contains at least two edges. By symmetry, we assume $y_1y_3 \in
  E(Y)$. By contracting two edges represented by thick curves, it
  becomes clear that $G_x$ contains $K_6$ as a minor.}
\label{fig:Zempty}
\end{center}
\end{figure}
Secondly, suppose $e(Z) = 0$. Now $Z$ is an $\alpha_x$-set and it
follows from Proposition~\ref{prop:baseCaseExtended}, that $G_x$
possess the structure as indicated in Figure~\ref{fig:Zempty}. By
contracting the edges $w_1z_3$ and $w_3 z_1$ of $G_x$ into two
distinct vertices $w_1'$ and $w_3'$, we find that the vertices $w_1',
w_2, w_3', y_1, y_3$ and $z_2$ induce a complete $6$-graph, and we are
done. Thus, in the following we shall be assuming $e(Z) \geq 1$.
Moreover, we shall distinguish between several cases depending on the
number of edges in $E(Y)$ and $E(Z)$. So far we have established
$e(Y) \geq 2$ and $2 \geq e(Z) \geq 1$. We shall often use the
fact that $\deg (u, G_x) \in \{5,6,8 \}$ for every vertex $u \in G_x$,
in particular, each vertex of $G_x$ can have at most three
non-neighbours in $G_x$ (excluding itself).
\begin{figure}
\begin{center}
\input{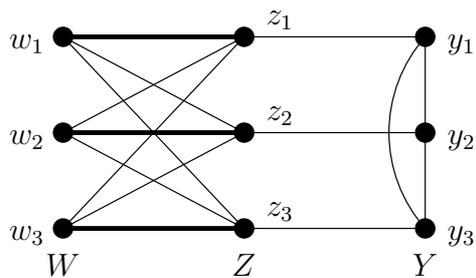}
\caption{The graph $G_x$ contains the graph depicted above as a
  subgraph. The thick curves indicate the edges to be contracted. By
  contracting two edges of $G_x$ as indicated above, it becomes
  obvious that $G_x$ contains $K_6$ as a minor.}
\label{fig:YcompletePlusMatching}
\end{center}
\end{figure}
\begin{itemize}
\item[1)] Suppose $e(Y) = 3$.
\item[1.1)] Suppose, moreover, that there is a matching $M$ of $Y$ into
  $Z$, say $M = \{y_1z_1, y_2z_2, y_3z_3 \}$. Now contracting the
  edges $w_i z_i$ ($i=1,2,3)$ into three distinct vertices a complete
  $6$-graph is obtained, and we are done (see
  Figure~\ref{fig:YcompletePlusMatching}).
\item[1.2)] Suppose that there is no matching of $Y$ into $Z$. Now it
  follows from Hall's Theorem~\cite[Th. 16.4]{MR2368647} that there
  exists some non-empty set $S \subseteq Y$ such that $e(S,Z) < |S|$
  (recall, that $e(S,Z)$ denotes the number of edges with one
  end-vertex in $S$ and the other end-vertex in $Z$). According to
  Proposition~\ref{prop:baseCaseExtended}, $e(S,Z) \geq 1$ for any
  non-empty $S \subseteq Y$.
\item[1.2.1)] Suppose that $e(Y,Z) = 1$, say $E(Y,Z) = \{ z_1 \}$.
  Now $y_1,y_2$ and $y_3$ are all non-neighbours of $z_2$ and $z_3$,
  and so both $z_2$ and $z_3$ must be adjacent to each other and to
  $z_1$, that is, $e(Z) = 3$, contradicting our assumption that
  $e(Z) \leq 2$. 
\item[1.2.2)] Suppose that $e(Y, Z) = 2$, say $E(Y,Z) = \{ z_1, z_2
  \}$. Now $y_1, y_2$ and $y_3$ are three non-neighbours of $z_3$, and
  so $z_3$ must be adjacent to both $z_2$ and $z_3$. Since
  $e(Z) \leq 2$, it must be the case that $z_1$ and $z_2$ are
  non-neighbours. Since no vertex of $G_x$ has precisely one
  non-neighbour, both $z_1$ and $z_2$ must have at least one
  non-neighbour in $Y$. By symmetry, we may assume that $y_1$ is a
  non-neighbour of $z_1$. Now $w_1, z_1$ and $z_3$ are three
  non-neighbours of $y_1$, and so $y_1$ cannot be a non-neighbour of
  $z_2$. It follows that $y_2$ or $y_3$ must be a non-neighbour of
  $z_2$. By symmetry, we may assume $y_2 z_2 \notin E(G)$. Now there
  may be no more edges missing in $G_x$, however, we assume that there
  are more edges missing, and show that $G_x$ remains contractible to
  $K_6$. Each of the vertices $y_1$ and $y_2$ has three non-neighbours
  specified, while $y_3$ already has two non-neighbours
  specified. Thus, the only possible hitherto undetermined missing
  edge must be either $y_3 z_1$ or $y_3 z_2$ (not both, since that
  would imply $y_3$ to have at least four non-neighbours). By
  symmetry, we may assume $y_3 z_2 \notin E(G)$. Now it is clear that $G_x$ is
  isomorphic to the graph depicted in Figure~\ref{fig:H17}, and so
  it follows from Lemma~\ref{lem:F} that $G$ is contractible to
  $K_7$.
\begin{figure}
\begin{center}
\input{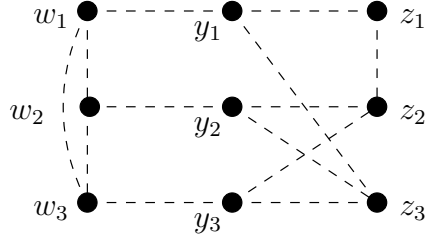}
\caption{
In case 1.2.2, the graph $G_x$ is isomorphic to the graph depicted
  above. Any edge which is not explicity indicated missing is
  present.}
\label{fig:H17}
\end{center}
\end{figure}
\item[1.2.3)] Suppose that $e(Y,Z) = 3$. Now, since there is no
  matching of $Y$ into $Z$ there must be some non-empty proper subset
  $S$ of $Y$ such that $|S| \leq 2$ and $e(S,Z) < |S|$. Recall,
  $e(S,Z) \geq 1$ for any non-empty subset $S$ of $Y$, and so it
  must be the case that $|S| = 2$ and $e(S,Z) = 1$, say $S = \{ y_1,
  y_2 \}$ and $E(S, Z) = \{ z_1 \}$. The assumption $e(Y, Z) = 3$
  implies that $y_3$ is adjacent to both $z_2$ and $z_3$. According to
  Proposition \ref{prop:baseCaseExtended}~(d), each vertex of $Y$ has
  a non-neighbour in $Z$, and so it must be the case that $y_3$ is not
  adjacent to $z_1$. Now, since $z_1$ has one non-neighbour in
  $V(G_x)\backslash \{z_1 \}$,
  Proposition~\ref{prop:missingNeighboursInGx}~(b) implies that it
  must have at least one other non-neighbour in $V(G_x) - z_1$. The
  only possible non-neighbours of $z_1$ in $V(G_x) \backslash \{ z_1,
  y_3 \}$ are $z_2$ and $z_3$, and, by symmetry, we may assume that
  $z_1$ and $z_2$ are not adjacent. Thus, $z_2$ is adjacent to neither
  $z_1$, $y_1$ nor $y_2$ and so $z_2$ must be adjacent to every vertex
  of $V(G_x) \backslash \{z_1, z_2, y_1, y_2 \}$, in particular, $z_2$
  is adjacent to $z_3$.  Thus, $G_x$ contains the graph depicted in
  Figure~\ref{fig:noMatchingS2NS1} as a subgraph. Now, by contracting
  the edges $w_1 z_1$, $w_2 y_1$ and $w_3 y_2$ of $G_x$ into three
  distinct vertices a complete $6$-graph is obtained.
\begin{figure}
\begin{center}
\input{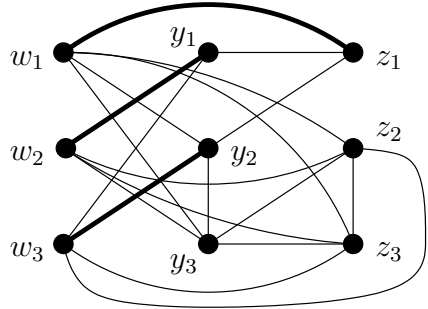}
\caption{The graph $G_x$ contains the graph depicted above as a
  subgraph. The thick curves indicate the edges to be contracted. By
  contracting three edges of $G_x$ as indicated above, it becomes
  obvious that $G_x$ contains $K_6$ as a minor.}
\label{fig:noMatchingS2NS1}
\end{center}
\end{figure}
\item[2)] Suppose $e(Y) = 2$, say $y_1y_2, y_2y_3 \in E(G)$.
\item[2.1)] Suppose that $e(Z) = 2$, say $z_1z_2, z_2z_3 \in E(G)$. 
\item[2.1.1)] Suppose that at least one of the edges $y_1 z_1$ or
  $y_3z_3$ are not in $E(G)$, say $y_1 z_1 \notin E(G)$. The vertex
  $y_1$ has three non-neighbours in $G_x$, namely $w_1, y_3$ and
  $z_1$. Thus, $y_1$ must be adjacent to both $z_2$ and $z_3$. We have
  determined the edges of $E(W)$, $E(Y)$ and $E(Z)$, and the edges
  joining vertices of $W$ with vertices of $Y \cup Z$. Moreover, $G_x$
  contains at least two edges joining vertices of $Y$ with vertices of
  $Z$, as indicated in Figure~\ref{fig:F34}~(a). It follows that $G_x$
  contains the graph depicted in Figure~\ref{fig:F34}~(b) as a
  subgraph. By contracting the edges $w_1 y_2$, $w_2 y_3$ and $w_3
  z_1$ of $G_x$ into three distinct vertices a complete $6$-graph is
  obtained, and so $G \geq K_7$.
\begin{figure}[htbp]
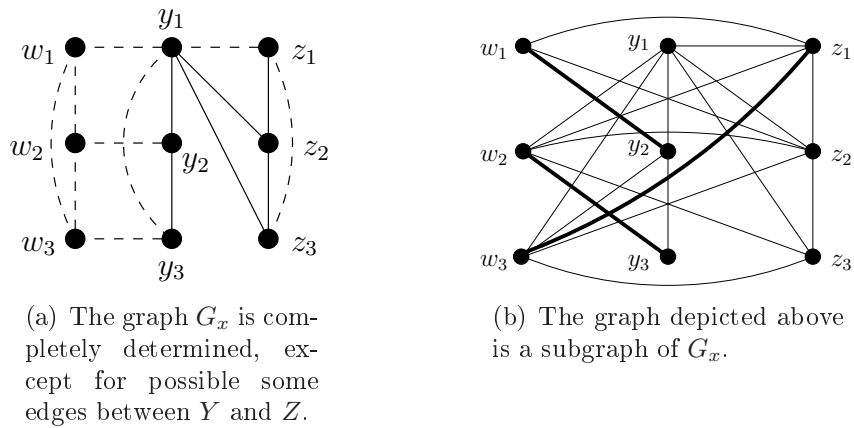

  \begin{center}
    \mbox{ \subfigure[The graph $G_x$ is completely determined, except
      for possible some edges between $Y$ and
      $Z$.]{\scalebox{1.0}{\input{F34a.pstex_t}}} \hspace{2cm}
      \subfigure[The graph depicted above is a subgraph of
      $G_x$.]{\scalebox{0.8}{\input{F34b.pstex_t}}} }
\caption{The case 2.1.1.}
    \label{fig:F34} \end{center}
\end{figure}
\item[2.1.2)] Suppose that both $y_1 z_1$ and $y_3 z_3$ are
  in $E(G)$. 
\item[2.1.2.1)]
Suppose that $y_1 z_2$ or $y_3 z_2$ is in $E(G)$, say $y_1 z_2 \in
E(G)$. In this case $G_x$ contains the graph depicted in
Figure~\ref{fig:Q1}~(a) as a subgraph, and so by contracting the edges
$w_1 y_2$, $w_2 y_3$ and $w_3 z_3$ into three distinct vertices a
complete $6$-graph is obtained.
\begin{figure}[htbp]
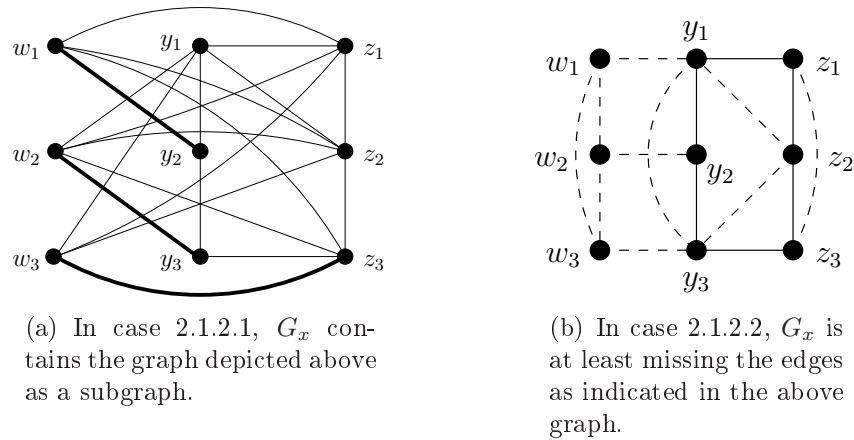

  \begin{center} \mbox{ \subfigure[In case 2.1.2.1, $G_x$ contains the
    graph depicted above as a
    subgraph.]{\scalebox{0.8}{\input{Q1.pstex_t}}} \hspace{2cm} \subfigure[In
    case 2.1.2.2, $G_x$ is at least missing the edges as indicated in
    the above graph.]{\scalebox{1.0}{\input{Q1b.pstex_t}}} }
\caption{The case 2.1.2.}
    \label{fig:Q1} \end{center}
\end{figure}
\item[2.1.2.2)]
Suppose that neither $y_1 z_2$ nor $y_3 z_2$ is in $E(G)$. Now $S
:= \{ y_1, z_2, y_3 \}$ is an independent set of $G_x$ and so,
according to Proposition~\ref{prop:baseCaseExtended}~(c), the vertex
$z_2$ has a private non-neighbour in $V(G_x) - S$ w.r.t.\ $S$, and, as
is easily seen from Figure~\ref{fig:Q1}~(b), the only possible
non-neighbour of $z_2$ in $V(G_x)$ is $y_2$.  The vertices $z_1$ and
$z_3$ are not adjacent, and so, according to
Proposition~\ref{prop:baseCaseExtended}~(b), each of them must have a
second non-neighbour. Since $y_1$ and $y_3$ already have three
non-neighbours specified, it follows that the only possible
non-neighbour of $z_1$ and $z_3$ is $y_2$, but if neither $z_1$ nor
$z_3$ are adjacent to $y_2$, then $y_2$ would have at least four
non-neighbours in $G_x$, a contradiction.
\item[2.2)] Suppose that $e(Z) = 1$, say $E(Z) = \{ z_1 z_3 \}$.
\item[2.2.1)]
Suppose that $y_2 z_2 \in E(G)$. Now at least one of the edges $y_1
z_2$ and $y_3 z_2$ is in $E(G)$, since otherwise $z_2$ would have at
least four non-neighbour. By symmetry, we may assume $y_1 z_2 \in
E(G)$. At least one of the edges $y_1 z_1$ and $y_1 z_3$ must be in
$E(G)$, since $y_1$ cannot have more than three non-neighbours. By
symmetry, we may assume $y_1 z_1 \in E(G)$ (see
Figure~\ref{fig:F5}~(a)). By contracting the edges $w_1 z_1$, $w_3
z_3$ and $y_2 y_3$ of $G_x$ into three distinct vertices we obtain a
complete $6$-graph (see Figure~\ref{fig:F5}~(b)), and, thus, $G \geq
K_7$.
\begin{figure}[htbp]
  \begin{center} \mbox{ \subfigure[The graph $G_x$ is completely
    determined, except for some edges between $Y$ and
    $Z$.]{\scalebox{1.0}{\input{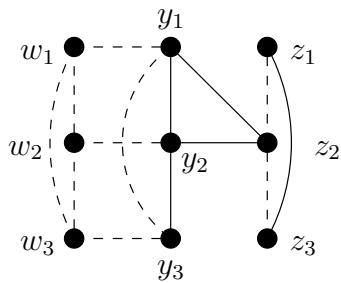}}} \hspace{2cm} \subfigure[The
    above graph is a subgraph of
    $G_x$.]{\scalebox{0.8}{\input{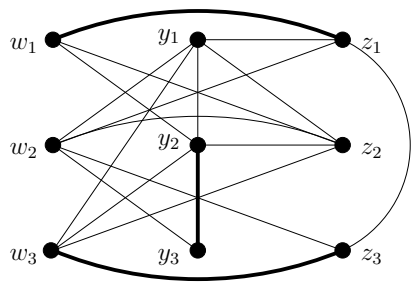}}} }
\caption{The case 2.2.1.}
    \label{fig:F5} \end{center}
\end{figure}
\item[2.2.2)]
Suppose that $y_2 z_2 \notin E(G)$. Each of the vertices $z_1$ and
$z_3$ has exactly one non-neighbour in $Z$, namely $z_2$, and so each
must have at least one non-neighbour in $Y$. If neither $z_1$ nor
$z_3$ were adjacent to $y_2$, then $y_2$ would have at least four
non-neighbours in $G_x$. Thus, at least one of $z_1$ and $z_3$ is not
adjacent to $y_1$ or $y_3$. By symmetry, we may assume that $y_1
z_1 \notin E(G)$. Now we need to determine the non-neighbour of $y_3$ in $Y$.
\item[2.2.2.1)]
Suppose that $y_2 z_3 \in E(G)$. Since $y_1$ already has three
non-neighbours, it must be the case that $y_3$ is a non-neighbour of
$z_3$ in $Y$. There may also be an edge joining $y_2$ and $z_1$, but
in any case$G_x$ contains the graph depicted in
Figure~\ref{fig:Q2}~(a) as a subgraph. Thus, by contracting the edges
$w_2 z_1$, $w_3 z_1$ and $y_1 z_2$ into three distinct vertices, we
find that $K_6 \leq G_x$.
\begin{figure}[htbp]
  \begin{center} \mbox{ \subfigure[In case 2.2.2.1, $G_x$ contains the
    graph depicted above as a
    subgraph.]{\scalebox{0.8}{\input{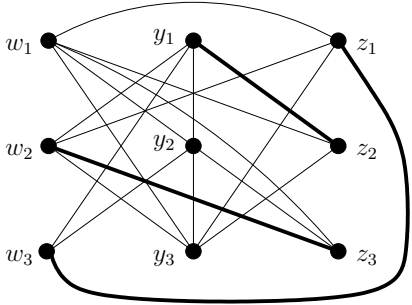}}} \hspace{2cm} \subfigure[In
    case 2.2.2.2, $G_x$ contains the graph depicted above as a
    subgraph.]{\scalebox{0.8}{\input{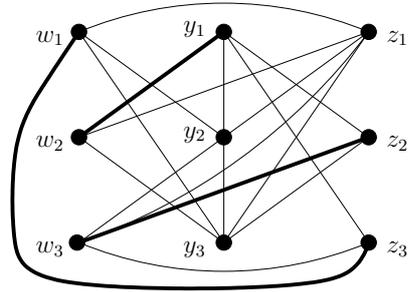}}} }
\caption{The case 2.2.2.}
    \label{fig:Q2} \end{center}
\end{figure}
\item[2.2.2.2)]
Suppose that $y_2 z_3 \notin E(G)$. In this case we find that $S := \{
y_2, z_2, z_3 \}$ is a maximum independent set in $G_x$ and so,
according to Proposition~\ref{prop:baseCaseExtended}~(c), each of the
vertices of $S$ has a private non-neighbour in $V(G_x) - S$
w.r.t.\ $S$. The vertices $w_1$, $y_3$ and $z_1$ are all non-neighbours
of $y_1$, and so $z_3$ cannot be a non-neighbour of $y_1$. It follows
that the non-neighbour of $z_3$ in $V(G_x) - S$ must be $y_3$. Now
each of the vertices of $Y$ has three non-neighbours, and so there can
be no further edges missing from $G_x$, that is, $G_x$ contains the
graph depicted in Figure~\ref{fig:Q2}~(b) as a subgraph.
\end{itemize}
This, finally, completes the case $\delta (G) = 9$, and so the proof
is complete.
\end{proof}
Obviously, if every $k$-chromatic graph for some fixed integer $k$ is
contractible to the complete $k$-graph, then every $\ell$-chromatic
graph with $\ell \geq k$ is contractible to the complete $k$-graph.
The corresponding result for \emph{double-critical} graphs is not
obviously true. However, for $k \leq 7$, it follows from
the aforementioned results and Corollary~\ref{cor:2347} that every
double-critical $\ell$-chromatic graph with $\ell \geq k$ is
contractible to the complete $k$-graph.
\begin{corollary}
  Every double-critical $k$-chromatic graph with $k \geq 7$ contains
  $K_7$ as a minor.
\label{cor:2347}
\end{corollary}
\begin{proof}
  Let $G$ denote an arbitray double-critical $k$-chromatic graph with
  $k \geq 7$. If $G$ is complete, then we are done. If $k = 7$, then
  the desired result follows from
  Theorem~\ref{th:doubleCriticalContractionCriticalSevenChrom}. If $k
  \geq 9$, then, according to Proposition~\ref{prop:minimumDegree},
  $\delta (G) \geq 10$ and so the desired result follows from a
  theorem of Gy{\H{o}}ri~\cite{MR652887} and Mader~\cite{MR0229550}.
  Suppose $k=8$ and that $G$ is non-complete. Then $\delta (G) \geq 9$.
  If $\delta (G) \geq 10$, then we are done and so we may assume
  $\delta (G) = 9$, say $\deg (x) = 9$. In this case it follows from
  Proposition~\ref{prop:structureOfGx} that the complement
  $\overline{G_x}$ consists of cycles (at least one) and isolated
  vertices (possibly none). An argument similar to the argument given
  in the proof of Theorem~\ref{th:DoubleCriticalChrom6} shows that
  $G_x$ is contractible to $K_6$. Since $x$ dominates every vertex of
  $V(G_x)$, then $G$ itself is contractible to $K_7$.
\end{proof}
The problem of proving that every double-critical $8$-chromatic graph
is contractible to $K_8$ remains open.
\section{Double-edge-critical-  and mixed-double-critical graphs}
\label{sec:mixed}
A natural variation on the theme of double-critical graphs is to
consider double-edge-critical graphs. A vertex-critical graph $G$ is
called \emph{double-edge-critical} if the chromatic number of $G$
decreases by at least two whenever two non-incident edges are removed
from $G$, that is,
\begin{equation}
  \chi (G - e_1 - e_2 ) \leq \chi(G) - 2 \textrm{ for any two non-incident edge } e_1, e_2 \in E(G)
\label{eq:928568376545}
\end{equation}
It is easily seen that $\chi (G - e_1 - e_2 )$ can never be strictly
less that $\chi (G) - 2$ and so we may require $\chi (G - e_1 -
e_2 ) = \chi(G) - 2$ in \eqref{eq:928568376545}. The only critical
$k$-chromatic graphs for $k \in \{1,2 \}$ are $K_1$ and $K_2$,
therefore we assume $k \geq 3$ in the following.
\begin{theorem}
  A graph $G$ is $k$-chromatic double-edge-critical if and only if it
  is the complete $k$-graph.
\label{th:doubleEdgeCritical}
\end{theorem}
\begin{proof}
  It is straightforward to verify that any complete graph is
  double-edge-critical. Conversely, suppose $G$ is a $k$-chromatic ($k
  \geq 3$) double-edge-critical graph. Then $G$ is connected. If $G$
  is a complete graph, then we are done. Suppose $G$ is not a complete
  graph. Then $G$ contains an induced $3$-path $P: wxy$. Since $G$ is
  vertex-critical, $\delta (G) \geq k-1 \geq 2$, and so $y$ is
  adjacent to some vertex $z$ is $V(G)\backslash \{w,x,y \}$. Now the
  edges $wx$ and $yz$ are not incident, and so $\chi (G - wx - yz) = k
  - 2$. Let $\varphi$ denote a $(k-2)$-colouring of $G - wx -
  yz$. Then the vertices $w$ and $x$ (and $y$ and $z$) are assigned
  the same colours, since otherwise $G$ would be
  $(k-1)$-colourable. We may assume that $\varphi$ assigns the colour
  $k-3$ to the vertices $w$ and $x$, and the colour $k-2$ to the
  vertices $y$ and $z$. Now define the $(k-1)$-colouring $\varphi'$
  such that $\varphi'(v) = \varphi(v)$ except $\varphi' (w) = k-1$ and
  $\varphi'(y) = k-1$. The colouring $\varphi'$ is a proper
  $(k-1)$-colouring, since $w$ and $y$ are non-adjacent in $G$. This
  contradicts the fact that $G$ is $k$-chromatic and therefore $G$
  must be a complete graph.
\end{proof}
A vertex-critical $k$-chromatic graph $G$ is called
\emph{mixed-double-critical} if for any vertex $x \in G$ and any edge
$e=uv \in E(G - x)$,
\begin{equation}
  \chi ( G - x - e ) \leq \chi(G) - 2
\label{eq:9276876545}
\end{equation}
\begin{theorem}
  A graph $G$ is $k$-chromatic mixed-double-critical if and only if it
  is the complete $k$-graph.
\label{th:mixed}
\end{theorem}
The proof of Theorem~\ref{th:mixed} is straightforward and similar to
the proof of Theorem~\ref{th:doubleEdgeCritical}.

\section*{Acknowledgment}
The authors wish to thank Marco Chiarandini and Steffen Elberg
Godskesen for creating computer programs for testing small graphs.
  
\bibliographystyle{plainnat}  
\bibliography{natbibFull}

\begin{thebibliography}{22}
\providecommand{\natexlab}[1]{#1}
\providecommand{\url}[1]{\texttt{#1}}
\expandafter\ifx\csname urlstyle\endcsname\relax
  \providecommand{\doi}[1]{doi: #1}\else
  \providecommand{\doi}{doi: \begingroup \urlstyle{rm}\Url}\fi

\bibitem[Balogh et~al.(2008)Balogh, Kostochka, Prince, and Stiebitz]{Balogh08}
J.~Balogh, A.~V. Kostochka, N.~Prince, and M.~Stiebitz.
\newblock The {E}rd{\H o}s-{L}ov\'asz tihany conjecture for quasi-line graphs,
  2008.
\newblock Submitted.

\bibitem[Bondy and Murty(2008)]{MR2368647}
J.~A. Bondy and U.~S.~R. Murty.
\newblock \emph{Graph Theory}, volume 244 of \emph{Graduate Texts in
  Mathematics}.
\newblock Springer, New York, 2008.

\bibitem[Brown and Jung(1969)]{MR0242718}
W.~G. Brown and H.~A. Jung.
\newblock On odd circuits in chromatic graphs.
\newblock \emph{Acta Math. Acad. Sci. Hungar.}, 20:\penalty0 129--134, 1969.

\bibitem[Dirac(1952)]{MR0045371}
G.~A. Dirac.
\newblock A property of {$4$}-chromatic graphs and some remarks on critical
  graphs.
\newblock \emph{J. London Math. Soc.}, 27:\penalty0 85--92, 1952.

\bibitem[Erd\H{o}s(1968)]{TihanyProblem2}
P.~Erd\H{o}s.
\newblock Problem 2.
\newblock In \emph{Theory of Graphs (Proc. Colloq., Tihany, 1966)}, page 361.
  Academic Press, New York, 1968.

\bibitem[Gallai(1964)]{Gallai63}
T.~Gallai.
\newblock Critical graphs.
\newblock In \emph{Theory of Graphs and its Applications (Proc. Sympos.
  Smolenice, 1963)}, pages 43--45. Publ. House Czechoslovak Acad. Sci., Prague,
  1964.

\bibitem[Gallai(1963)]{MR0188100}
T.~Gallai.
\newblock Kritische {G}raphen. {II}.
\newblock \emph{Magyar Tud. Akad. Mat. Kutat\'o Int. K\"ozl.}, 8:\penalty0
  373--395 (1964), 1963.

\bibitem[Gy{\H{o}}ri(1982)]{MR652887}
E.~Gy{\H{o}}ri.
\newblock On the edge numbers of graphs with {H}adwiger number {$4$} and {$5$}.
\newblock \emph{Period. Math. Hungar.}, 13\penalty0 (1):\penalty0 21--27, 1982.

\bibitem[Hadwiger(1943)]{MR0012237}
H.~Hadwiger.
\newblock \"{U}ber eine {K}lassifikation der {S}treckenkomplexe.
\newblock \emph{Vierteljschr. Naturforsch. Ges. Z\"urich}, 88:\penalty0
  133--142, 1943.

\bibitem[Jakobsen(1971)]{MR0340108}
I.~T. Jakobsen.
\newblock A homomorphism theorem with an application to the conjecture of
  {H}adwiger.
\newblock \emph{Studia Sci. Math. Hungar.}, 6:\penalty0 151--160, 1971.

\bibitem[Jensen and Toft(1995)]{JensenToft95}
T.~R. Jensen and B.~Toft.
\newblock \emph{Graph Coloring Problems}.
\newblock Wiley-Interscience Series in Discrete Mathematics and Optimization.
  John Wiley \& Sons Inc., New York, 1995.

\bibitem[Kawarabayashi and Toft(2005)]{MR2141662}
K.~Kawarabayashi and B.~Toft.
\newblock Any 7-chromatic graph has {$K\sb 7$} or {$K\sb {4,4}$} as a minor.
\newblock \emph{Combinatorica}, 25\penalty0 (3):\penalty0 327--353, 2005.

\bibitem[Kostochka and Stiebitz(2008)]{KostochkaStiebitz2008}
A.~V. Kostochka and M.~Stiebitz.
\newblock Partitions and edge colourings of multigraphs.
\newblock \emph{Electron. J. Combin.}, 15\penalty0 (1):\penalty0 Note 25, 4,
  2008.

\bibitem[Krusenstjerna-Hafstr{\o}m and Toft(1981)]{MR634546}
U.~Krusenstjerna-Hafstr{\o}m and B.~Toft.
\newblock Some remarks on {H}adwiger's conjecture and its relation to a
  conjecture of {L}ov\'asz.
\newblock In \emph{The theory and applications of graphs (Kalamazoo, Mich.,
  1980)}, pages 449--459. Wiley, New York, 1981.

\bibitem[Mader(1968)]{MR0229550}
W.~Mader.
\newblock Homomorphies\"atze f\"ur {G}raphen.
\newblock \emph{Math. Ann.}, 178:\penalty0 154--168, 1968.

\bibitem[Mozhan(1987)]{0688.05026}
N.N. Mozhan.
\newblock {On doubly critical graphs with the chromatic number five.}
\newblock \emph{Metody Diskretn. Anal.}, 46:\penalty0 50--59, 1987.

\bibitem[Neumann~Lara(1982)]{MR693366}
V.~Neumann~Lara.
\newblock The dichromatic number of a digraph.
\newblock \emph{J. Combin. Theory Ser. B}, 33\penalty0 (3):\penalty0 265--270,
  1982.

\bibitem[Royle(2008)]{RoyleCriticalGraphs08}
G.~Royle.
\newblock Gordon {R}oyle's small graphs.
\newblock 25-06, 2008.
\newblock URL
  \url{http://people.csse.uwa.edu.au/gordon/remote/graphs/index.html}.

\bibitem[Stiebitz(1988)]{MR1221590}
M.~Stiebitz.
\newblock On {$k$}-critical {$n$}-chromatic graphs.
\newblock In \emph{Combinatorics (Eger, 1987)}, volume~52 of \emph{Colloq.
  Math. Soc. J\'anos Bolyai}, pages 509--514. North-Holland, Amsterdam, 1988.

\bibitem[Stiebitz(1987)]{MR882614}
M.~Stiebitz.
\newblock {$K\sb 5$} is the only double-critical {$5$}-chromatic graph.
\newblock \emph{Discrete Math.}, 64\penalty0 (1):\penalty0 91--93, 1987.

\bibitem[Toft(1995)]{MR1373659}
B.~Toft.
\newblock Colouring, stable sets and perfect graphs.
\newblock In \emph{Handbook of Combinatorics, Vol.\ 1}, pages 233--288.
  Elsevier, Amsterdam, 1995.

\bibitem[Toft(1996)]{MR1411244}
B.~Toft.
\newblock A survey of {H}adwiger's conjecture.
\newblock \emph{Congr. Numer.}, 115:\penalty0 249--283, 1996.

\end{thebibliography}

\end{document}